\documentclass[12pt]{amsart}
\usepackage{amssymb}
\usepackage{mathtools}
\usepackage[english]{babel}

\usepackage{tikz}
\usetikzlibrary{calc}
\usetikzlibrary[patterns]

\usepackage{epsfig}
\numberwithin{equation}{section}

\def\Re{{\sf Re}\,}
\def\Im{{\sf Im}\,}

\def\1#1{\overline{#1}}
\def\2#1{\widetilde{#1}}
\def\3#1{\widehat{#1}}
\def\4#1{\mathbb{#1}}
\def\5#1{\frak{#1}}
\def\6#1{{\mathcal{#1}}}

\newcommand{\R}{\mathbb R}
\newcommand{\Ha}{\mathbb H}

\newcommand{\C}{\mathbb C}

\newcommand{\D}{\mathbb D}
\newcommand{\oD}{\overline{\mathbb D}}

\newcommand{\N}{\mathbb N}

\newcommand{\Arg}{\mathrm{Arg}}

\def\Re{{\sf Re}\,}
\def\Im{{\sf Im}\,}

\newcommand{\strip}{\mathbb{S}}

\newcommand{\mcite}[1]{\csname b@#1\endcsname}

\theoremstyle{theorem}

\def\Re{{\sf Re}\,}
\def\Im{{\sf Im}\,}

\newtheorem{theorem}{Theorem}[section]
\newtheorem{lemma}[theorem]{Lemma}
\newtheorem{proposition}[theorem]{Proposition}
\newtheorem{corollary}[theorem]{Corollary}

\theoremstyle{definition}
\newtheorem{definition}[theorem]{Definition}

\theoremstyle{remark}
\newtheorem{remark}[theorem]{Remark}

\numberwithin{equation}{section}

\title[Non-tangential limits and the slope of trajectories]{Non-tangential limits and the slope of trajectories of holomorphic semigroups of the unit disc}

\author[F. Bracci]{Filippo Bracci$^\ast$}
\address{F. Bracci: Dipartimento di Matematica, Universit\`a di Roma ``Tor Vergata", Via della Ricerca
Scientifica 1, 00133, Roma, Italia.} \email{fbracci@mat.uniroma2.it}

\author[M. D. Contreras]{Manuel D. Contreras$^\dag$}

\author[S. D\'{\i}az-Madrigal]{Santiago D\'{\i}az-Madrigal$^\dag$}
\address{M. D. Contreras, S. D\'{\i}az-Madrigal: Camino de los Descubrimientos, s/n\\
Departamento de Matem\'{a}tica Aplicada~II and IMUS\\ Universidad de Sevilla\\ Sevilla,
41092\\ Spain.}\email{contreras@us.es} \email{madrigal@us.es}

\author[H. Gaussier]{Herv\'e Gaussier$^\ddag$}
\address{H. Gaussier: Univ. Grenoble Alpes, CNRS, IF, F-38000 Grenoble, France.}\email{herve.gaussier@univ-grenoble-alpes.fr}


\thanks{$^\ast$ Partially supported by the MIUR Excellence Department Project awarded to the
Department of Mathematics, University of Rome Tor Vergata, CUP E83C18000100006}

\thanks{$^\dag$ Partially supported by the \textit{Ministerio
de Econom\'{\i}a y Competitividad} and the European Union (FEDER) MTM2015-63699-P and  by \textit{La Consejer\'{\i}a de Educaci\'{o}n y Ciencia de la Junta de Andaluc\'{\i}a}.}

\thanks{$^\ddag$ Partial supported by ERC ALKAGE}

\long\def\REM#1{\relax}

\begin{document}
\maketitle

\begin{abstract} Let $\Delta\subsetneq \C$ be a simply connected domain, let $f:\D \to \Delta$ be a Riemann map and let  $\{z_k\}\subset \Delta$ be a compactly divergent sequence. Using Gromov's hyperbolicity theory, we show that $\{f^{-1}(z_k)\}$  converges non-tangentially to a point of $\partial \D$ if and only if there exists a simply connected domain $U\subsetneq \C$ such that $\Delta\subset U$ and $\Delta$ contains a tubular hyperbolic neighborhood  of a geodesic of $U$ and $\{z_k\}$ is eventually contained in a smaller tubular hyperbolic neighborhood of   the same geodesic. As a consequence we show that if $(\phi_t)$ is a non-elliptic  semigroup of holomorphic self-maps of $\D$ with Koenigs function $h$ and $h(\D)$ contains a vertical Euclidean sector, then $\phi_t(z)$ converges to the Denjoy-Wolff point non-tangentially for every $z\in \D$ as $t\to +\infty$. Using new localization results for the hyperbolic distance, we also construct an example of a parabolic semigroup which converges non-tangentially to the Denjoy-Wolff point but it is oscillating, in the sense  that the slope of the trajectories is not a single point.
\end{abstract}

\tableofcontents

\section{Introduction}

The notion of {\sl non-tangential} limit is very important in geometric function theory. A sequence $\{z_n\}\subset \D:=\{z\in \C: |z|<1\}$ converges non-tangentially to a point $\sigma\in \partial \D$ if it converges to $\sigma$ and  it is eventually contained in a Stolz region of vertex $\sigma$, that is, if it is eventually contained in the set $\{z\in \D: |\sigma-z|<R(1-|z|)\}$ for some $R>1$.

\medskip

{\bf Question.} Let $\Delta\subsetneq \C$ be a simply connected domain and let $f:\D \to \Delta$ be a Riemann map. Let $\{z_n\}\subset \Delta$ be a compactly divergent sequence, {\sl i.e.}, with no accumulation points in $\Delta$. How can one decide  if  $\{f^{-1}(z_n)\}$  converges non-tangentially to a point $\sigma$ by looking only at geometric properties of $\Delta$?

\medskip

The first aim of this paper is to give an answer to this question using the hyperbolic distance (a similar question for orthogonal convergence has been settled in \cite{BCDG}).

The first observation is that, if $\gamma:[0,+\infty)\to \D$ is a geodesic for the hyperbolic distance $\omega$  of $\D$ parameterized by arc length, then there exists $\sigma\in \partial \D$ such that $\lim_{t\to+\infty}\gamma(t)=\sigma\in \partial \D$. Moreover, for every $R>0$, the set $S_\D(\gamma, R):=\{z\in \D: \omega(z, \gamma([0,+\infty))<R\}$, which we call a {\sl hyperbolic sector} around $\gamma$ of amplitude $R$, is equivalent to a Stolz region at $\sigma$. Therefore, a compactly divergent sequence $\{z_n\}\subset \D$ converges non-tangentially to $\sigma$ if and only if it is eventually contained in a hyperbolic sector around a geodesic converging to $\sigma$. Given a simply connected domain $U\subset \C$ and a biholomorphism $f:\D \to U$, the map $f$ is an isometry between the hyperbolic distance $k_U$ of $U$ and $\omega$, thus the previous property  is invariant under biholomorphisms and gives a first answer to the previous question. However, such a conclusion is not useful in practice, because knowing geodesics and the hyperbolic distance of a simply connected domain is almost equivalent to knowing the Riemann map of that domain.

However, using the Gromov hyperbolicity theory, we  prove the following result (see Theorem \ref{Thm:nec-suff-non-tg}):

\begin{theorem}\label{Thm:nec-suff-non-tg-intro}
Let $\Delta\subsetneq \C$ be a simply connected domain and let $f:\D\to \Delta$ be a Riemann map. Let $\{z_n\}\subset \Delta$ be a compactly divergent sequence, {\sl i.e.}, with no accumulation points in $\Delta$. Then $\{f^{-1}(z_n)\}$ converges non-tangentially to a point $\sigma\in \partial \D$   if and only if there exist a simply connected domain $U\subsetneq \C$, a geodesic  $\gamma:[0,+\infty)\to U$ of $U$ such that $\lim_{t\to+\infty}k_U(\gamma(t), \gamma(0))=+\infty$ and $R>R_0>0$ such that
\begin{enumerate}
\item $S_U(\gamma, R):=\{w\in U:\, k_{U}(w,\gamma ([0,+\infty)))<R\}\subset\Delta\subseteq U$,
\item  there exists $n_0\geq 0$ such that $z_n\in S_U(\gamma, R_0)$   for all $n\geq n_0$.
\end{enumerate}
\end{theorem}

A simple consequence of the previous theorem is that if $\Delta$ is a simply connected domain contained in an upper half-plane and containing a {\sl vertical Euclidean sector} $p+\{z\in \C: \Im z> k| \Re z|\}$, for some $p\in \C$ and $k>0$,  then $f^{-1}(p+it)$ converges non-tangentially to a boundary point of $\D$ as $t\to+\infty$.

Another interesting consequence is that if $\Delta$ is a simply connected domain starlike at infinity (that is, $\Delta+it\subset \Delta$ for all $t\geq 0$), that contains a vertical Euclidean sector, then the curve $[0,+\infty)\ni t\mapsto f^{-1}(f(z)+it)$ converges non-tangentially to some point of $\partial \D$ as $t\to+\infty$ for every $z\in \D$. In fact, it can be shown that such a curve is a uniform quasi-geodesic in the sense of Gromov.

The latter fact has an interesting application to the study of {\sl one-parameter continuous semigroups} of holomorphic self-maps of $\D$---or, for short, semigroups in $\D$. A semigroup in $\D$ is a continuous homomorphism of the real semigroup $[0,+\infty)$ endowed with the Euclidean topology to the semigroup under composition of holomorphic self-maps of $\D$ endowed with the topology of uniform convergence on compacta. Semigroups in $\D$ have been intensively studied (see, {\sl e.g.},\cite{Abate,AhElReSh99,Berkson-Porta, ES,Shb,Siskakis-tesis}). It is known that, if $(\phi_t)$ is a semigroup in $\D$, which is not a group of hyperbolic rotations, then there exists a unique $\tau\in \oD$, the {\sl Denjoy-Wolff point} of $(\phi_t)$ such that $\lim_{t\to +\infty}\phi_t(z)=\tau$, and the convergence is uniform on compacta. In case $\tau\in \D$, the semigroup is called elliptic. Non-elliptic semigroups can be divided into three types: hyperbolic, parabolic of positive hyperbolic step and parabolic of zero hyperbolic step. It is known (see \cite{CD, CDP, EYRS}) that if $(\phi_t)$ is a hyperbolic semigroup then the trajectory $t\mapsto \phi_t(z)$ always converges non-tangentially to its Denjoy-Wolff point as $t\to +\infty$ for every $z\in \D$, while, if it is parabolic of positive hyperbolic step then $\phi_t(z)$ always converges tangentially to its Denjoy-Wolff point as $t\to +\infty$ for every $z\in \D$.

In case of parabolic semigroups of zero hyperbolic step, the behavior of trajectories can be rather wild. All the trajectories have the same {\sl slope}, that is the cluster set of $\Arg(1-\overline{\tau}\phi_t(z))$ as $t\to +\infty$---which is a compact subset of $[-\pi/2,\pi/2]$---does not depend on $z\in \D$ (see \cite{CD, CDP}). In many cases this slope is just a point, but in  \cite{Bet, CDG}, examples are constructed  such  that the slope is the full interval $[-\pi/2,\pi/2]$.

Recall  (see, {\sl e.g.}, \cite{Abate, BrAr, Cowen, BCD, EKRS}) that  $(\phi_t)$ is a parabolic semigroup in $\D$ of zero hyperbolic step if and only if  there exists a univalent function $h$, the {\sl Koenigs function} of $(\phi_t)$, such that $h(\D)$ is starlike at infinity, $h(\phi_t(z))=h(z)+it$ for all $t\geq 0$ and $z\in \D$, and for every $w\in \C$ there exists $t_0\geq 0$ such that $w+it_0\in h(\D)$. The triple $(\C, h, z\mapsto z+it)$ is called a {\sl canonical model} for $(\phi_t)$ and it is essentially unique. A straigthforward consequence of our previous discussion is the following (see Proposition \ref{sector-implies-convergnt}):

\begin{corollary}\label{non-tg-intro}
Let $(\phi_t)$ be a parabolic semigroup of zero hyperbolic step with Denjoy-Wolff point $\tau\in \partial \D$. Assume that $h(\D)$ contains a vertical Euclidean sector. Then for every $z\in \D$ the trajectory $\phi_t(z)$ converges non-tangentially to $\tau$ as $t\to+\infty$. In other words, the slope of $(\phi_t)$ is a set $[a,b]$ with $-\pi/2<a\leq b<\pi/2$.
\end{corollary}

The condition of  $h(\D)$ containing a vertical Euclidean sector is  not necessary for having non-tangential convergence of the orbits: let $\alpha>1$ and let $Z_\alpha:=\{z\in \C: |\Re z|^\alpha<\Im z\}$, a parabola-like open set. Since $Z_\alpha$ is simply connected and starlike at infinity, if $f:\D \to Z_\alpha$ is a Riemann map,  $\phi_t(z):=f^{-1}(f(z)+it)$,  $z\in \D$, is a semigroup whose canonical model is $(\C, f, z\mapsto z+it)$, hence, it is  parabolic of zero hyperbolic step. Since $Z_\alpha$ is symmetric with respect to the imaginary axis, it follows that $f^{-1}(it)$, $t>0$, is a geodesic in $\D$. Therefore $\phi_t(f^{-1}(i))$  converges radially to the Denjoy-Wolff point as $t\to +\infty$.

Despite the previous example, it turns out that for every $\alpha>1$ there exists a parabolic semigroup $(\phi_t^\alpha)$ of zero hyperbolic step with Koenigs function $h_\alpha$ such that $Z_\alpha\subset h_\alpha(\D)$ but $\phi_t^\alpha(z)$ does not converge non-tangentially to the Denjoy-Wolff point (see Proposition~\ref{Prop:tang-conve-with-para}).

Furthermore, using Corollary \ref{non-tg-intro}, we are able to construct a (rather explicit in terms of the canonical model) example of a parabolic semigroup with zero hyperbolic step whose trajectories converge non-tangentially to the Denjoy-Wolff point but are oscillating (see Proposition \ref{Prop:example-non-tg-osc}):

\begin{proposition}\label{main-para}
There exists a parabolic semigroup $(\phi_t)$ of zero hyperbolic step in $\D$ such that for every $z\in \D$, the slope of $(\phi_t)$ is $[a,b]$ with $-\pi/2<a<b<\pi/2$.
\end{proposition}

Our technique does not allow to prescribe exactly the values of $a, b$. In \cite{Bet} (see also \cite{K} for details) it is remarked that, with a slight modification of the technique of harmonic measure theory used by the author in order to construct parabolic semigroups with slope $[-\pi/2,\pi/2]$, it is possible to construct examples of parabolic semigroups having slope $[a,b]$ {\sl for every} $-\pi/2<a<b<\pi/2$.

The proof of the latter proposition is quite involved and the techniques we use might be interesting in their own right. The main new tools we introduce and exploit in our construction are``good boxes'' (see Section \ref{good}). These are open subsets of simply connected domains where one can estimates hyperbolic distance and displacement of geodesics using the corresponding objects for strips.

The plan of the paper is the following. In Section \ref{geo}, we recall the notion of geodesics and Gromov's quasi-geodesics in simply connected domains and state some results we need in the paper. In Section \ref{loc}, we collect some known (and some possibly new) results of localization  for the hyperbolic metric and the hyperbolic distance in simply connected domains. In Section \ref{hypse}, we prove Theorem \ref{Thm:nec-suff-non-tg-intro} and Corollary \ref{non-tg-intro}. In Section \ref{good} we introduce ``good boxes'' and prove the results about geodesics and hyperbolic metric we mentioned above. Finally, in Section \ref{Traj}, we prove Proposition \ref{main-para} and  Proposition \ref{Prop:tang-conve-with-para}.

\smallskip

We thank the referees for many useful comments which improved the original manuscript.

\medskip

{\bf Notations.} In this paper we will freely make use of Carath\'eodory's prime end theory (see, {\sl e.g.},\cite{CL, Pommerenke, Pommerenke2}). In particular, recall that  every simply connected domain $\Delta\subsetneq \C$ has a Carath\'eodory boundary $\partial_C\Delta$ given by the union of the prime ends of $\Delta$. The set $\widehat{\Delta}:=\Delta\cup \partial_C\Delta$ can be endowed with  the Carath\'eodory topology. For an open set $U\subset \Delta$, we let $U^\ast$ be the union of $U$ with every prime end of $\Delta$ for which there exists a representing null chain which is eventually contained in $U$. The Carath\'eodory topology is the topology generated by all open sets $U$ of $\Delta$ and the sets $U^\ast$. It is known that $\oD$ with the Euclidean topology is homeomorphic to $\widehat{\D}$ and, if $f:\D \to \Delta$ is a Riemann map, then $f$ extends to a homeomorphism $\hat{f}: \widehat{\D}\to \widehat{\Delta}$. In this way, every point $\sigma\in \partial \D$ corresponds to a unique prime end $\underline{x}_\sigma\in \partial_C\D$ and, via $f$, to a unique prime end $\hat{f}(\underline{x}_\sigma)\in \partial_C\Delta$.

\smallskip

We denote by $\C_\infty$ the Riemann sphere. If $\Delta\subset \C$ is a domain, we denote by $\partial_\infty \Delta$ its boundary in $\C_\infty$. Note that $\partial_\infty \Delta=\partial \Delta$ in case $\Delta$ is bounded, otherwise $\partial_\infty \Delta=\partial \Delta\cup\{\infty\}$.

Finally, we denote by $\omega(z,w)$ the hyperbolic distance between $z$ and $w\in \D$.

\section{Geodesics and quasi-geodesics in simply connected domains}\label{geo}

Let $\Delta\subsetneq \C$ be a simply connected domain. We denote by $\kappa_\Delta$ the infinitesimal metric in $\Delta$, that is,  for $z\in \Delta$, $v\in \C$, we let
\[
\kappa_\Delta(z;v):=\frac{|v|}{f'(0)},
\]
where $f:\D\to \Delta$ is the Riemann map such that $f(0)=z$, $f'(0)>0$. The hyperbolic distance $k_\Delta$ in $\Delta$ is defined for $z, w\in \Delta$ as
\[
k_\Delta(z,w):=\inf \int_0^1 \kappa_\Delta(\gamma(t);\gamma'(t))dt,
\]
where the infimum is taken over all piecewise $C^1$-smooth curve $\gamma:[0,1]\to \Delta$ such that $\gamma(0)=z, \gamma(1)=w$.

It is well known that, for all $z,w\in \D$,
\[
\omega(z,w):=k_\D(z,w)=\frac{1}{2}\log \frac{1+\left|\frac{z-w}{1-\overline{z}w} \right|}{1-\left|\frac{z-w}{1-\overline{z}w} \right|}.
\]

Let  $-\infty<a<b<+\infty$ and let $\gamma:[a,b]\to \Delta$ be a piecewise $C^1$-smooth curve. For $a\leq s\leq t\leq b$, we define the {\sl hyperbolic length of $\gamma$ in $\Delta$} between $s$ and $t$ as
\[
\ell_\Delta(\gamma;[s,t]):=\int_s^t \kappa_\Delta(\gamma(u);\gamma'(u))du.
\]
In case $s=a$ and $t=b$, we will simply write
\[
\ell_\Delta(\gamma):=\ell_\Delta(\gamma;[a,b]).
\]

\begin{definition}
Let $\Delta\subsetneq \C$ be a simply connected domain.  A    $C^1$-smooth curve $\gamma:(a,b)\to \Delta$, $-\infty\leq a<b\leq +\infty$ such that $\gamma'(t)\neq 0$ for all $t\in (a,b)$ is called a {\sl geodesic} of $\Delta$  if for every $a< s\leq t< b$,
\[
\ell_\Delta(\gamma;[s,t])=k_\Delta(\gamma(s), \gamma(t)).
\]
Moreover, if $z,w\in \Delta$ and there exist $a<s<t<b$ such that $\gamma(s)=z$ and $\gamma(t)=w$, we say that $\gamma|_{[s,t]}$ is a geodesic which joins $z$ and $w$.

With a slight abuse of notation, we call geodesic also the image of $\gamma$ in $\Delta$.
\end{definition}

Using Riemann maps and the invariance of hyperbolic metric and distance under the action of biholomorphisms, we have the following result:

\begin{proposition}\label{Prop:geodesic-in-simply}
Let $\Delta\subsetneq \C$ be a simply connected domain. Let $-\infty\leq a<b\leq +\infty$.
\begin{enumerate}
\item If $\eta:(a,b) \to \Delta$ is a geodesic,  then
\[
\eta(a):=\lim_{t\to a^+}\eta(t), \quad \eta(b):=\lim_{t\to b^-}\eta(t)
\]
 exist as limits in the Carath\'eodory topology of $\Delta$. Moreover, if $\eta(a), \eta(b)\in \Delta$ then
 \[
k_\Delta(\eta(a),\eta(b))=\lim_{\epsilon\to 0^+}\ell_{\Delta}(\eta;[a+\epsilon,b-\epsilon]).
 \]
\item If $\eta:(a,b) \to \Delta$ is a geodesic such that $\eta(a), \eta(b)\in \partial_C \Delta$, then $\eta(a)\neq \eta(b)$.
\item For any $z,w\in \widehat{\Delta}$, $z\neq w$, there exists a real analytic geodesic $\gamma:(a,b)\to \Delta$ such that $\gamma(a)=z$ and $\gamma(b)=w$. Moreover, such a geodesic is essentially unique, namely, if $\eta:(\tilde a, \tilde b)\to \Delta$ is another geodesic joining $z$ and $w$, then $\gamma([a,b])=\eta([\tilde a,\tilde b])$ in $\widehat{\Delta}$.
\item If  $\gamma:(a,b)\to \Delta$ is a geodesic such that either $\gamma(a)\in \Delta$ or $\gamma(b)\in \Delta$ (or both), then there exists a geodesic $\eta:(\tilde a,\tilde b)\to \Delta$ such that $\eta(\tilde a), \eta(\tilde b)\in \partial_C \Delta$ and such that $\gamma([a,b])\subset \eta([\tilde a, \tilde b])$ in $\widehat{\Delta}$.
\item If  $\gamma:(a,b)\to \Delta$ is a geodesic such that  $\gamma(a)\in \partial_C\Delta$ then the cluster set $\Gamma(\gamma,a)$ is equal to $\Pi(\gamma(a))$, the principal part of the prime end $\gamma(a)$ (and similarly for $b$ in case $\gamma(b)\in  \partial_C\Delta$).
\end{enumerate}
\end{proposition}

Given a simply connected domain, it is in general a hard task to find geodesics. The aim of this section is to recall a powerful method due to Gromov to localize geodesics via simpler curves which are called quasi-geodesics.

\begin{definition}
Let $\Delta\subsetneq \C$ be a simply connected domain. Let $A\geq 1$ and $B\geq 0$. A piecewise $C^1$-smooth curve $\gamma:[a,b]\to \Delta$, $-\infty<a<b<+\infty$, is a {\sl $(A,B)$-quasi-geodesic} if  for every $a\leq s\leq t\leq b$,
\[
\ell_\Delta(\gamma; [s,t])\leq A k_\Delta(\gamma(s),\gamma(t))+B.
\]
\end{definition}

The importance of quasi-geodesics is partly justified by the following shadowing lemma (see, {\sl e.g.} \cite{Ghys}), known also as ``geodesics' stability lemma'':

\begin{theorem}[Gromov's shadowing lemma]\label{Gromov}
For every $A\geq 1$ and $B\geq 0$ there exists $\delta=\delta(A,B)>0$ with the following property. Let  $\Delta\subsetneq \C$ be any simply connected domain. If $\gamma:[a,b]\to \Delta$ is a $(A,B)$-quasi-geodesic, then there exists a geodesic $\tilde\gamma:[\tilde a, \tilde b]\to \Delta$ such that $\tilde\gamma(\tilde a)=\gamma(a)$, $\tilde\gamma(\tilde b)=\gamma(b)$ and for every $u\in [a,b]$ and $v\in [\tilde a, \tilde b]$,
\[
k_\Delta(\gamma(u), \tilde\gamma([\tilde a, \tilde b]))<\delta, \quad k_\Delta(\tilde\gamma(v),\gamma([ a,  b]))< \delta.
\]
\end{theorem}

The following result is a consequence of Gromov's shadowing lemma and follows by standard normality arguments:

\begin{corollary}\label{Cor:shadow}
Let  $\Delta\subsetneq \C$ be a simply connected domain. Let $\gamma:[0,+\infty)\to \Delta$ be a piecewise $C^1$-smooth curve such that $\lim_{t\to +\infty}k_\Delta(\gamma(0), \gamma(t))=+\infty$ and there exist $A\geq 1$, $B\geq 0$ such that for every fixed $T>0$ the curve $[0,T]\ni t\mapsto \gamma(t)$ is a $(A,B)$-quasi-geodesic. Then there exists a prime end $\underline{x}\in \partial_C\Delta$ such that $\gamma(t)\to \underline{x}$ in the Carath\'eodory topology of $\Delta$ as $t\to +\infty$. Moreover, there exists $\epsilon>0$ such that, if $\eta:[0,+\infty)\to \Delta$ is the geodesic of $\Delta$ parameterized by arc length such that $\eta(0)=\gamma(0)$ and $\lim_{t\to+\infty}\eta(t)=\underline{x}$ in the Carath\'eodory topology of $\Delta$, then, for every $t\in [0,+\infty)$,
\begin{equation*}
k_\Delta(\gamma(t), \eta([0,+\infty)))<\epsilon, \quad k_\Delta(\eta(t), \gamma([0,+\infty)))<\epsilon.
\end{equation*}
\end{corollary}

\section{Localization of  hyperbolic metric and  hyperbolic distance}\label{loc}

In this section we prove a localization result which allows to get information on the hyperbolic metric and hyperbolic distance of a simply connected domain in a portion of the domain itself.

We start with the notion of totally geodesic subsets:

\begin{definition}
Let $\Delta\subsetneq \C$ be a simply connected domain. A  domain $U\subset \Delta$ is said to be {\sl totally geodesic} in $\Delta$ if for every $z, w\in U$ the geodesic of $\Delta$ joining $z$ and $w$ is contained in $U$.
\end{definition}

We need  the following lemma:

\begin{lemma}\label{Lem:total-geo-disc}
Let $\Delta\subsetneq \C$ be a simply connected domain. Let $\gamma:\R\to \Delta$ be a geodesic parameterized by arc length. Then $\Delta\setminus\gamma(\R)$ consists of two simply connected components which are totally geodesic in $\Delta$.
\end{lemma}
\begin{proof}
Let $f:\D\to \Delta$ be a biholomorphism. Then, $f^{-1}\circ \gamma:\R \to \D$ is a geodesic parameterized by arc length. Up to pre-composing with an automorphism of $\D$, we can assume that $f^{-1}(\gamma(\R))=(-1,1)$.
Consider $\D^{+}:=\{\zeta\in \D: \Re \zeta>0\}$. Since the geodesic in $\D$ joining two points $z,w\in \D$ is the arc of a circle containing $z, w$ and meeting $\partial \D$ orthogonally, it is easy to see that $\D^+$ is totally geodesic in $\D$. A similar argument shows that  $\D^{-}:=\{\zeta\in \D: \Re \zeta<0\}$ is totally geodesic. Moving back to $\Delta$ via $f$ and recalling that $f$ is an isometry for the hyperbolic distance, we have the result.
\end{proof}

Now we can state and prove a localization result for the hyperbolic metric and the hyperbolic distance:

\begin{theorem}[Localization Lemma]\label{Thm:localiz}
Let $\Delta\subsetneq \C$ be a simply connected domain. Let $p\in \partial_C\Delta$ and let $U^\ast$ be an  open set in $\widehat{\Delta}$ which contains $p$. Assume that $U^\ast\cap \Delta$ is simply connected. Let $C>1$. Then there exists an open neighborhood  $V^\ast\subset U^\ast$ of $p$ such that for every $z, w\in V^\ast\cap \Delta$ and all $v\in \C$,
\begin{equation}\label{Eq:localization1}
\kappa_\Delta(z;v)\leq \kappa_{U^\ast\cap \Delta}(z;v)\leq C \kappa_{\Delta}(z;v),
\end{equation}
\begin{equation}\label{Eq:localization2}
k_\Delta(z,w)\leq k_{U^\ast\cap \Delta}(z,w)\leq C k_\Delta(z,w).
\end{equation}
In particular, if $\Delta$ is a Jordan domain then for every $\sigma\in \partial_\infty\Delta$, for every $U\subset \C_\infty$ open set such that $\sigma\in U$ and $U\cap \Delta$ is simply connected, and every $C>1$, there exists an open neighborhood $V\subset U$ of $\sigma$ such that \eqref{Eq:localization1} and \eqref{Eq:localization2} hold (with $U^\ast=U$ and $V^\ast=V$).
\end{theorem}
\begin{proof}
The inequalities on the left in \eqref{Eq:localization1} and \eqref{Eq:localization2} follow immediately from the decreasing properties of the infinitesimal metric and of the distance.

As for the inequalities on the right, it is enough to prove them for $\Delta=\D$. The identity map extends to a homeomorphism $\Phi$ between $\widehat{\D}$ and $\overline{\D}$. Hence, there exists an open set (in the Euclidean topology) $W\subset \C$ such that $\sigma:=\Phi(p)\in W$ and $\Phi(U^\ast)=W\cap \overline \D$.  Since by hypothesis $U^\ast\cap \D$ is simply connected, then $\Phi(U^\ast\cap \D)=W\cap \D$ is simply connected as well.

Now we use ideas probably well known to experts. However, we give a sketch here for the reader's convenience. First, one can prove that given $R>0$ such that $(\tanh R)^{-1}<C$, there exists  an open set $X\subset W$, $\sigma\in X$, such that for every $z\in X\cap \D$ the hyperbolic disc $D^{hyp}(z,R):=\{w \in \D:k_{\D}(z,w) < R\}$ is contained in $W\cap \D$. This implies immediately that for all $z\in X\cap \D$ and $v\in \C$,
\begin{equation}\label{Eq:prima-mericaW}
\kappa_{W\cap \D}(z;v)\leq \kappa_{D^{hyp}(z,R)}(z;v)=(\tanh R)^{-1} \kappa_\D(z;v)<C \kappa_\D(z;v).
\end{equation}

Then, one can find $\epsilon\in (0,\pi/4)$ in such a way that, if   $\gamma:(-\infty,+\infty)\to \D$ is the geodesic in $\D$ parameterized by arc length such that $\lim_{t\to-\infty}\gamma(t)=e^{\epsilon i}\sigma $ and $\lim_{t\to+\infty}\gamma(t)=e^{-\epsilon i}\sigma$ then $\gamma(\R)\subset X$.

By Lemma \ref{Lem:total-geo-disc}, $\D\setminus \gamma(\R)$ is the union of two simply connected components. Since $\overline{\gamma(\R)}$ does not contain $\sigma$, it follows that $\sigma$ belongs to the closure of one and only one of the connected components of  $\D\setminus \gamma(\R)$. Call   $Y$  such a component.  By Lemma \ref{Lem:total-geo-disc}, $Y$ is totally geodesic in $\D$.  Therefore, for every $z,w\in Y$, the geodesic $\eta:[0,1]\to \D$ of $\D$ such that $\eta(0)=z, \eta(1)=w$ is contained in $Y\subset X\cap \D$. Hence, by \eqref{Eq:prima-mericaW},
\begin{equation*}
\begin{split}
k_{W\cap \D}(z,w)&\leq \ell_{W\cap \D}(\eta;[0,1])=\int_0^1\kappa_{W\cap \D}(\eta(t);\eta'(t))dt\\&\leq C\int_0^1\kappa_\D(\eta(t);\eta'(t))dt=Ck_\D(z,w).
\end{split}
\end{equation*}
By the arbitrariness of $z,w$, setting $V^\ast:=\Phi^{-1}(\tilde{Y}\cap \overline{\D})$, where $\tilde{Y}$ is any open set in $\C$ such that $\tilde{Y}\cap \D=Y$, we are done.

Finally, if $\Delta$ is a Jordan domain, the result follows since  $\overline{\Delta}^\infty$ and $\widehat{\Delta}$ are homeomorphic.
\end{proof}

If $\Omega\subset \C$ is a domain, for $z\in \Omega$, we let
\[
\delta_\Omega(z):=\inf_{w\in \C\setminus \Omega} |z-w|,
\]
the Euclidean distance from $z$ to the boundary $\partial \Omega$.

\begin{theorem}\label{Thm:Distance-Lemma-inf} Let $\Delta\subsetneq \C$ be a simply connected domain. Then for every $z\in \Delta$ and $v\in \C$,
\[
 \frac{|v|}{4\delta_\Delta(z)}\leq \kappa_\Delta(z;v)\leq \frac{|v|}{\delta_\Delta(z)}.
\]
Moreover, if $\Delta$ is convex, $ \kappa_\Delta(z;v)\geq  \frac{|v|}{2\delta_\Delta(z)}$ for every $z\in \Delta$ and $v\in \C$.
\end{theorem}
\begin{proof}[Sketch of the proof]
The lower estimate follows form the Koebe $1/4$-Theorem. The upper estimate follows at once since the Euclidean disc of center $z$ and radius $\delta_\Delta(z)$ is contained in $\Delta$.

In case $\Delta$ is convex, take $z\in \Delta$ and let $p\in \partial \Delta$ be a point such that $|p-z|=\delta_\Delta(z)$. By convexity, $\Delta$ is contained in a half-plane whose boundary is a separating line for $\Delta$ at $p$. From this the lower estimate follows at once.
\end{proof}

Integrating the previous estimates, one has:

\begin{theorem}[Distance Lemma]\label{Thm:Distance-Lemma} Let $\Delta\subsetneq \C$ be a simply connected domain. Then for every $w_1, w_2\in \Delta$,
\[
 \frac{1}{4} \log \left(1+\frac{|w_1-w_2|}{\min\{\delta_\Delta(w_1), \delta_\Delta(w_2)\}} \right)\leq k_\Delta(w_1,w_2)\leq \int_{\Gamma}\frac{|dw|}{\delta_\Delta(w)},
\]
where $\Gamma$ is any piecewise $C^1$-smooth curve in $\Delta$ joining $w_1$ to $w_2$.

In case $\Delta$ is convex, one can replace $1/4$ with $1/2$ in the left-hand side of the previous inequality.
\end{theorem}

\section{Hyperbolic sectors and non-tangential limits}\label{hypse}

The aim of this section is to provide an intrinsic way to define non-tangential limits in simply connected domains. More precisely, the question we settle here is the following: let $\Delta \subsetneq \C$ be a simply connected domain and $f:\D\to \Delta$ a Riemann map. Let
$\{z_n\}\subset \Delta$ be a sequence such that $\{f^{-1}(z_n)\}$ converges to $\sigma\in \partial \D$. How is it possible to determine whether $\{f^{-1}(z_n)\}$ converges non-tangentially to $\sigma$ by looking at the geometry of $\Delta$?

We start with a definition which allows to extend the notion of non-tangential limit to any simply connected domain:

\begin{definition}
Let  $\Delta\subsetneq \C$ be a simply connected domain. Let $\gamma:(a,+\infty)\to \Delta$, $a\geq -\infty$, be a geodesic with the property that $\lim_{t\to +\infty}k_\Delta(\gamma(t),\gamma(t_{0}))=+\infty$, for some $t_{0}\in (a,+\infty)$. A {\sl hyperbolic sector around $\gamma$ of amplitude $R>0$} is
\[
S_\Delta(\gamma, R):=\{w\in \Delta: k_\Delta(w, \gamma((a,+\infty)))<R\}.
\]
\end{definition}

Now we aim to give a description of hyperbolic sectors. To this aim, it is useful to move our considerations to the right half-plane, where actual computations turn out to be easier.

Let $\Ha:=\{z \in \C : \Re(z) > 0\}$. The map $z\mapsto \frac{1+z}{1-z}$ is a biholomorphism between $\D$ and $\Ha$. Hence, a direct computation shows that for $z,w\in \Ha$, $v\in \C$
\[
\kappa_\Ha(z;v)=\frac{|v|}{2\Re z}, \quad k_\Ha(z,w)=\frac{1}{2}\log \frac{1+\left|\frac{z-w}{z+\overline{w}} \right|}{1-\left|\frac{z-w}{z+\overline{w}} \right|}.
\]
Moreover, since $\D$ and $\Ha$ are biholomorphic via a Moebius transformation,  it follows easily that the geodesics in $\Ha$ are either intervals contained in semi-lines in $\Ha$ parallel to the real axis, or arcs in $\Ha$ of circles intersecting orthogonally the imaginary axis.

\begin{lemma}\label{Lem:hyper-semipiano}
Let $\beta\in (-\frac{\pi}{2},\frac{\pi}{2})$.
\begin{enumerate}
\item Let  $0<\rho_0<\rho_1$ and let $\Gamma:=\{\rho e^{i\beta}: \rho_0\leq \rho\leq \rho_1\}$. Then, $\displaystyle{\ell_{\Ha}(\Gamma)=\frac{1}{2\cos \beta}\log\frac{\rho_1}{\rho_0}}$.
\item Let $\rho_0, \rho_1>0$. Then, $\displaystyle{k_{\Ha}(\rho_0,\rho_1e^{i\beta})-k_{\Ha}(\rho_0,\rho_1)\geq \frac{1}{2}\log\frac{1}{\cos \beta}.}$
\item Let $\rho_0>0$ and $\alpha\in (-\frac{\pi}{2},\frac{\pi}{2})$. Then, $(0,+\infty)\ni \rho\mapsto k_{\Ha}(\rho e^{i\alpha},\rho_0e^{i\beta})$ has a minimum at $\rho=\rho_0$, it is  increasing for $\rho>\rho_0$ and  decreasing for $\rho<\rho_0$.
\item Let $\theta_0, \theta_1\in (-\frac{\pi}{2},\frac{\pi}{2})$ and $\rho>0$. Then $k_{\Ha}(\rho e^{i\theta_0},\rho e^{i\theta_1})=k_{\Ha}(e^{i\theta_0},e^{i\theta_1})$. Moreover, $k_\Ha(1,e^{i\theta})=k_\Ha(1,e^{-i\theta})$ for all $\theta\in [0,\pi/2)$ and $[0, \pi/2)\ni \theta\mapsto k_\Ha(1,e^{i\theta})$ is strictly increasing.
\item Let $\beta_0,\beta_1\in (-\frac{\pi}{2},\frac{\pi}{2})$ and $0<\rho_0<\rho_1$. Then $\displaystyle{
k_{\Ha}(\rho_0e^{i\beta_0},\rho_1e^{i\beta_1})\geq k_{\Ha}(\rho_0,\rho_1)}$.
\end{enumerate}
\end{lemma}
\begin{proof}
(1) Setting $\gamma(\rho):=\rho e^{i\beta}$, we have
\[
\ell_\Ha(\Gamma)=\ell_\Ha(\gamma;[\rho_0,\rho_1])=\int_{\rho_0}^{\rho_1}\frac{1}{2\Re \rho e^{i\beta}}d\rho=\frac{1}{2\cos \beta}\log\frac{\rho_1}{\rho_0}.
\]
In particular, since for $\beta=0$, $\Gamma$ is a geodesic of $\Ha$, $\ell_\Ha(\Gamma;[\rho_0, \rho_1])=k_\Ha(\rho_0, \rho_1)$.

(2) We have,
\begin{equation*}
\begin{split}
k_{\Ha}(\rho_0,\rho_1e^{i\beta})&=\frac{1}{2}\log \frac{\left(1+\left\vert \frac{\rho_1e^{i\beta}-\rho_0}{\rho_1e^{i\beta}+\rho_0}\right\vert\right)^2}{1-\left\vert \frac{\rho_1e^{i\beta}-\rho_0}{\rho_1e^{i\beta}+\rho_0}\right\vert^2}=\frac{1}{2}\log \frac{\left(\left\vert{\rho_1e^{i\beta}+\rho_0}\right\vert+\left\vert \rho_1e^{i\beta}-\rho_0\right\vert\right)^2}{\left\vert {\rho_1e^{i\beta}+\rho_0}\right\vert^2-\left\vert {\rho_1e^{i\beta}-\rho_0}\right\vert^2}\\& = \frac{1}{2}\log \frac{\rho_0^2+\rho_1^2+\sqrt{\rho_1^4+\rho_0^4-2\rho_0^2\rho_1^2\cos(2\beta)}}{2\rho_0\rho_1\cos\beta}.
\end{split}
\end{equation*}
Assume $\rho_0\leq \rho_1$ and  set $x=\frac{\rho_0}{\rho_1}$ (in case $\rho_0>\rho_1$, set $x=\frac{\rho_1}{\rho_0}$). Hence,
\[
k_{\Ha}(\rho_0,\rho_1e^{i\beta})-k_{\Ha}(\rho_0,\rho_1)= \frac{1}{2}\log \frac{1+x^2+\sqrt{1+x^4-2x^2\cos(2\beta)}}{2\cos \beta}.
\]
Since the numerator inside the logarithm is strictly increasing in $x$ and $x\in (0,1]$, the estimate follows.

(3) We have
\begin{equation*}
k_{\Ha}(\rho e^{i\alpha},\rho_0e^{i\beta})=\frac{1}{2}\log \frac{1+\left\vert \frac{\rho_0e^{i\beta}-\rho e^{i\alpha}}{\rho_0e^{i\beta}+\rho e^{-i\alpha}}\right\vert}{1-\left\vert \frac{\rho_0e^{i\beta}-\rho e^{i\alpha}}{\rho_0e^{i\beta}+\rho e^{-i\alpha}}\right\vert}.
\end{equation*}
Since the derivative  of $[0,1)\ni x\mapsto \frac{1}{2}\log \frac{1+x}{1-x}$ is strictly positive, it is enough to prove the statement for the function
\[
(0,+\infty)\ni \rho\mapsto \frac{|\rho_0e^{i\beta}-\rho e^{i\alpha}|^2}{|\rho_0e^{i\beta}+\rho e^{-i\alpha}|^2}=\frac{\rho_0^2+\rho^2-2\rho\rho_0\cos(\beta-\alpha)}{\rho_0^2+\rho^2+2\rho\rho_0\cos(\beta+\alpha)},
\]
and this follows immediately from a direct computation.

(4) From a straightforward computation from the very definition of hyperbolic distance in $\Ha$, we have
\[
k_{\Ha}(\rho e^{i\theta_0},\rho e^{i\theta_1})=\frac{1}{2}\log \frac{1+\frac{|e^{i\theta_0}-e^{i\theta_1}|}{|e^{i\theta_0}+e^{-i\theta_1}|}}{1-\frac{|e^{i\theta_0}-e^{i\theta_1}|}{|e^{i\theta_0}+e^{-i\theta_1}|}}=k_{\Ha}( e^{i\theta_0}, e^{i\theta_1}).
\]
This proves the first part of the statement. Alternatively, this follows from the fact that the multiplication by $\rho$ is a biholomorphism of $\Ha$. Next, since
\[
[0,\pi/2)\ni \theta\mapsto \left|\frac{e^{i\theta}-1}{e^{i\theta}+1}\right|=\sqrt{\frac{1-\cos \theta}{1+\cos \theta}},
\]
is strictly increasing, using the fact that $(0,1)\ni x\mapsto \frac{1}{2}\log \frac{1+x}{1-x}$ is strictly increasing in $x$ and from the very definition of $k_{\Ha}(1, e^{i\theta})$ it follows that $[0, \pi/2)\ni \theta\mapsto k_\Ha(1,e^{i\theta})$ is strictly increasing. Moreover, the previous formula also shows that $k_{\Ha}(1, e^{i\theta})=k_{\Ha}(1, e^{-i\theta})$ for all $\theta\in [0,\pi/2)$.

(5) Using the fact that $(0,1)\ni x\mapsto \frac{1}{2}\log \frac{1+x}{1-x}$ is strictly increasing in $x$ and from the very definition of $k_{\Ha}$, it is enough to prove that
\[
\frac{|e^{i\beta_0}\rho_0-e^{i\beta_1}\rho_1|}{|e^{i\beta_0}\rho_0+e^{-i\beta_1}\rho_1|}\geq \frac{\rho_1-\rho_0}{\rho_0+\rho_1}.
\]
Setting $a:=\rho_0^2+\rho_1^2$ and $b=2\rho_0\rho_1$, and taking the square in the previous inequality, this amounts to show that
\[
\frac{a-b\cos(\beta_0-\beta_1)}{a+b\cos(\beta_0+\beta_1)}\geq \frac{a-b}{a+b}.
\]
After simple computations, this is equivalent to
\[
\cos\beta_1\cos\beta_0+\frac{b}{a}\sin\beta_1\sin\beta_0\leq 1.
\]
Since $b\leq a$, the result follows.
\end{proof}

Now we describe the shape of a hyperbolic sector in the half-plane. We need a definition:

\begin{definition}
Let $\beta\in (0,\pi)$ and $r_0\in [0,+\infty)$, let
\[
V(\beta, r_0):=\{\rho e^{i\theta}: \rho>r_0,  |\theta|< \beta\},
\]
be a {\sl horizontal sector} of angle $2\beta$ symmetric with respect to the real axis and with height $r_0$.
\end{definition}

\begin{lemma}\label{Lem:hyper-sector-inH}
Let $\gamma:[0,+\infty)\to \Ha$ be a geodesic such that $\gamma([0,+\infty))=[r_0, +\infty)$ and $\gamma(0)=r_0$   for some $r_0>0$. Then for every $R>0$ there exists
$\beta\in (0,\pi/2)$, with $k_\Ha(1,e^{i\beta})=R$, such that
\begin{equation}\label{Stolz-hyperb}
S_\Ha(\gamma, R)=V(\beta, r_0)\cup D^{hyp}_\Ha(r_0, R),
\end{equation}
where $D^{hyp}_\Ha(r_0, R):=\{w\in \Ha: k_\Ha(r_0, w)<R\}$ is the hyperbolic disc in $\Ha$ of center $r_0$ and radius $R$.
\end{lemma}

\begin{proof}
Let $w\in \Ha$, $w=\rho e^{i\theta}$ for some $\rho>0$ and $\theta\in (-\pi/2,\pi/2)$. Hence, by Lemma \ref{Lem:hyper-semipiano}(3),
\[
k_\Ha(w,(0,+\infty))=k_\Ha(\rho e^{i\theta}, \rho)=k_\Ha(e^{i\theta},1).
\]
Let $\beta\in (0,\pi/2)$ be such that $k_\Ha(1,e^{i\beta})=R$. Therefore, given $\rho>0$, by Lemma \ref{Lem:hyper-semipiano}(4) and the previous equalities,   $k_\Ha(\rho e^{i\theta}, (0,+\infty))<R$ if and only if $|\theta|<\beta$. This implies at once that $V(\beta, r_0)\subset S_\Ha(\gamma, R)$.

Moreover, let $w\in D^{hyp}_\Ha(r_0, R)$. Hence, $M:=k_\Ha(r_0, w)<R$. Let $r\in (r_0, +\infty)$ be such that $k_\Ha(r,r_0)<R-M$. Hence, by the triangle inequality,
\[
k_\Ha(w, r)\leq k_\Ha (w,r_0)+k_\Ha(r_0,r)<M+R-M=R,
\]
proving that $w\in S_\Ha(\gamma, R)$. Therefore, $V(\beta, r_0)\cup D^{hyp}_\Ha(r_0, R)\subset S_\Ha(\gamma, R)$.

Now, let $w=\rho e^{i\theta}\in S_\Ha(\gamma, R)$ with $\rho>0$ and $\theta\in (-\pi/2, \pi/2)$. If $\rho>r_0$, by Lemma \ref{Lem:hyper-semipiano}(3) and (4), it follows immediately that $w\in V(\beta, r_0)$. If $\rho\leq r_0$, the condition $w\in S_\Ha(\gamma, R)$ implies that there exists $r>r_0$ such that $k_\Ha(w,r)<R$. Hence,  by Lemma \ref{Lem:hyper-semipiano}(3),  $k_\Ha(\rho e^{i\theta}, r_0)<k_\Ha(\rho e^{i\theta}, r)<R$ and $w\in  D^{hyp}_\Ha(r_0, R)$. This proves that $S_\Ha(\gamma, R)\subset V(\beta, r_0)\cup D^{hyp}_\Ha(r_0, R)$.
\end{proof}

As a consequence, we have the following characterization of non-tangential convergence:

\begin{proposition}\label{Prop:conv-nt-sc}
Let $\Delta\subsetneq \C$ be a simply connected domain and let $f:\D \to \Delta$ be a Riemann map. Let $\{z_n\}\subset \Delta$ be a compactly divergent sequence. Then $\{f^{-1}(z_n)\}$ converges non-tangentially to $\sigma\in \partial \D$ if and only if there exist $R>0$ and a geodesic $\gamma:[0,+\infty)\to \Delta$ such that $\lim_{t\to +\infty}\gamma(t)=\hat{f}(\underline{x}_\sigma)$ in the Carath\'eodory topology of $\Delta$ and $\{z_n\}$ is eventually contained in $S_\Delta(\gamma, R)$. Here, $\hat{f}:\widehat{\D}\to \widehat{\Delta}$ is the homeomorphism induced by $f$ and $\underline{x}_\sigma\in \partial_C \D$ is the prime end corresponding to $\sigma$ under $f$.
\end{proposition}

\begin{proof}
Since the condition  that $\{z_n\}$ is eventually contained in $S_\Delta(\gamma, R)$ is invariant under isometries for the hyperbolic distance and $f$ is an isometry between $\omega$ and $k_\Delta$, it is enough to prove the statement for $\Delta=\Ha$ and a Cayley transform $f:\D\to \Ha$ which maps $\sigma $ to $\infty$. Hence, $\{z_n\}\subset \D$ converges non-tangentially to $\infty$ if and only if $\{z_n\}$ is eventually contained in a horizontal sector in $\Ha$. The result follows then at once by Lemma \ref{Lem:hyper-sector-inH}.
\end{proof}

The previous result allows to talk of non-tangential limits in simply connected domains, but, from a practical point of view, it is not very useful since  the characterization of hyperbolic sectors in a general simply connected domain is a very hard task. Still, we will see how, using localization, one can obtain useful conclusions. We start with   the following localization result for hyperbolic sectors:

\begin{lemma}\label{Lem:Stolz-qg}
Let $\Delta\subsetneq \C$ be a simply connected domain and let $\gamma:[0,+\infty)\to \Delta$ be a geodesic such that $\lim_{t\to +\infty}k_\Delta(\gamma(0), \gamma(t))=+\infty$. Let $R>0$. Then, for every $0<R_0<R$ there exists $C>1$ such that
\begin{enumerate}
\item for every $z\in S_\Delta(\gamma, R_0)$ and $v\in \C$,
\begin{equation}\label{eq:estima-sector}
\kappa_\Delta(z;v)\leq \kappa_{S_\Delta(\gamma, R)}(z;v)\leq C \kappa_\Delta(z;v),
\end{equation}
\item for every $z, w\in S_\Delta(\gamma, R_0)$,
\begin{equation}\label{eq:estima-sector2}
k_\Delta(z, w)\leq k_{S_\Delta(\gamma, R)}(z, w)\leq C k_\Delta(z,w).
\end{equation}
\end{enumerate}
\end{lemma}
\begin{proof} The left hand side inequalities follow at once  since $S_\Delta(\gamma, R)\subset \Delta$.

As for the right hand side inequalities in \eqref{eq:estima-sector} and \eqref{eq:estima-sector2}, since univalent maps are isometries for the hyperbolic distance, we can assume that $\Delta=\Ha$ and $\gamma([0,+\infty))=[1,+\infty)$.

By Lemma \ref{Lem:hyper-sector-inH},
\[
S:=S_\Ha(\gamma, R)=V(\beta, 1)\cup D^{hyp}_\Ha(1, R),
\]
 for some $\beta\in (0,\pi/2)$, and $S_\Ha(\gamma, R_0)=V(\beta', 1)\cup D^{hyp}_\Ha(1, R_0)$ for some $\beta'\in (0,\beta)$.

Taking into account that $\overline{D^{hyp}_\Ha(1, R_0)}\subset D^{hyp}_\Ha(1, R)$, it follows at once that $S_\Ha(\gamma, R_0)\cap\{w\in S: |z|\leq M\}$ is relatively compact in $S$ for every $M>1$. Therefore, given $M>1$ there exists $C$ (which depends on $M$) such that \eqref{eq:estima-sector} holds for every $z\in S_\Ha(\gamma, R_0)\cap\{w\in S: |z|\leq M\}$ and every $v\in \C$.

Fix  $M>1$  such that $\delta_S(z)=\delta_{V(\beta, 1)}(z)$ for all $z\in V(\beta',M)$.  By the previous argument, we only need to prove that \eqref{eq:estima-sector} holds for $z\in V(\beta', M)$. Let $z\in V(\beta', M)$ and $v\in \C\setminus\{0\}$. By Theorem \ref{Thm:Distance-Lemma-inf},
\begin{equation}\label{Eq:estima-k-hyp-sect1}
\frac{\kappa_S(z;v)}{\kappa_\Ha(z;v)}\leq 4\frac{\delta_\Ha(z)}{\delta_S(z)}=4\frac{\delta_\Ha(z)}{\delta_{V(\beta, 1)}(z)}.
\end{equation}
Now, let $z\in V(\beta', M)$ and let $q_z\in \partial V(\beta,1)$ be such that $|z-q_z|=\delta_{V(\beta, 1)}(z)$. If we write $z=\rho e^{i\theta}$ with $\rho>M$ and $|\theta|<\beta'$, assuming $\theta\geq 0$ (the case $\theta<0$ is similar), a simple computation shows that
\[
q_z-z=\rho\cos \beta \cos \theta (\tan\theta-\tan\beta)(\sin\beta-i\cos\beta).
\]
Hence,
\[
\delta_{V(\beta, 1)}(z)=|q_z-z|=\rho \cos \beta \cos \theta (\tan\beta-\tan\theta)\geq \rho \cos \beta \cos \theta (\tan\beta-\tan\beta').
\]
Since $\delta_\Ha(z)=\Re z=\rho\cos\theta$, we have
\[
\frac{\delta_\Ha(z)}{\delta_{V(\beta, 1)}(z)}\leq \frac{1}{\cos\beta(\tan\beta-\tan\beta')},
\]
and  the right hand side inequality in \eqref{eq:estima-sector} follows at once from \eqref{Eq:estima-k-hyp-sect1}.

We are left to prove the right hand side inequality in \eqref{eq:estima-sector2}.  To this aim, we claim that $S_\Ha(\gamma, R_0)$ is totally geodesic in $\Ha$.

Assuming the claim for the moment, let $z,w \in S_\Ha(\gamma, R_0)$ and let $\eta:[0,1]\to \Ha$ be a geodesic such that $\eta(0)=z$ and $\eta(1)=w$. By the claim, $\eta([0,1])\subset S_\Ha(\gamma, R_0)$. Hence, by \eqref{eq:estima-sector},
\[
k_{S_\Ha(\gamma, R)}(z,w)\leq \int_0^1\kappa_{S_\Ha(\gamma, R)}(\eta(t);\eta'(t))dt \leq C \int_0^1\kappa_{\Ha}(\eta(t);\eta'(t))dt=C k_\Ha(z,w),
\]
and the right hand side inequality in \eqref{eq:estima-sector2} follows.

Let us  prove the claim.  Since geodesics of $\Ha$ are either contained in half lines parallel to the real axis or in arcs of circles which intersect orthogonally the imaginary axis, it is clear that $V(\beta',0)$ is totally geodesic in $\Ha$. Next, since $\{\zeta\in \C: |\zeta|=r\}\cap \Ha$ is a geodesic in $\Ha$ for all $r>0$, it follows by Lemma \ref{Lem:total-geo-disc} that $\{w\in \Ha: |w|>1\}$ is totally geodesic in $\Ha$, hence, $V(\beta',1)=V(\beta',0)\cap \{w\in \Ha: |w|>1\}$ is totally geodesic in $\Ha$.

Moreover, $D^{hyp}_\Ha(1, R_0)$ is totally geodesic in $\Ha$ --- this can be easily seen by proving that any hyperbolic disc in $\D$ centered at $0$ is totally geodesic and using a Cayley transform to move to $\Ha$.

Therefore, we only have to show that if $z\in D^{hyp}_\Ha(1, R_0)\setminus V(\beta',1)$ and $w\in V(\beta',1)\setminus D^{hyp}_\Ha(1, R_0)$, the geodesic $\eta:[0,1]\to \Ha$ for $\Ha$ such that $\eta(0)=z$ and $\eta(1)=w$ is contained in $V(\beta',1)\cup D^{hyp}_\Ha(1, R_0)$.

To  this aim, we first observe that
$D^{hyp}_\Ha(1, R_0)\subset V(\beta',0)$. Indeed, if $\rho e^{i\theta}\in D^{hyp}_\Ha(1, R_0)$ for some $\rho>0$ and $\theta\in (-\pi/2, \pi/2)$, then by Lemma \ref{Lem:hyper-semipiano}(3),
\[
k_\Ha(\rho, \rho e^{i\theta})\leq k_\Ha(1, \rho e^{i\theta})<R_0.
\]
This, together with  Lemma \ref{Lem:hyper-semipiano}(4) and Lemma \ref{Lem:hyper-sector-inH}, proves that
\[
k_\Ha(1,  e^{i\theta})=k_\Ha(\rho,\rho e^{i\theta})<R_0=k_\Ha(1, e^{i\beta'}),
\]
and hence $|\theta|<\beta'$. That is,  $\rho e^{i\theta}\in V(\beta',0)$. Therefore, since $V(\beta',0)$ is totally geodesic in $\Ha$,
\begin{equation}\label{Eq:sta-inV-eq1}
\eta([0,1])\subset V(\beta',0).
\end{equation}
Hence, if $\eta([0,1])\not\subset V(\beta',1)\cup D^{hyp}_\Ha(1, R_0)$, then, by \eqref{Eq:sta-inV-eq1},  there exists $s\in (0,1)$ such that $|\eta(s)|<1$ and $\eta(s)\not\in D^{hyp}_\Ha(1, R_0)$. Now, the arc $(-\beta', \beta')\ni \theta\mapsto e^{i\theta}$ is  contained in $D^{hyp}_\Ha(1, R_0)$ by Lemma \ref{Lem:hyper-semipiano}(4), and divides $V(\beta',0)$ into two connected components, which are $V(\beta',1)$ and $V(\beta',0)\setminus \overline{V(\beta',1)}$.  Since $\eta([0,1])$ is connected, there exists $s'\in (s,1)$ such that $|\eta(s')|=1$---hence, $\eta(s')\in D^{hyp}_\Ha(1, R_0)$. But then, $\eta|_{[0,s']}$ is a geodesic in $\Ha$ which joins $z, \eta(s')\in D^{hyp}_\Ha(1, R_0)$ but it is not contained in $D^{hyp}_\Ha(1, R_0)$, contradicting the fact that $D^{hyp}_\Ha(1, R_0)$ is totally geodesic in $\Ha$. Therefore, $\eta([0,1])\subset V(\beta',1)\cup D^{hyp}_\Ha(1, R_0)$ and the claim follows.
\end{proof}

\begin{remark}\label{Rem:disco-in-sector-qg}
The last part of the proof of the previous lemma shows in particular that $D^{hyp}_\Ha(1, R_0)\subset V(\beta',0)$, where $k_\Ha(1, e^{i\beta'})=R_0$. Note that, by Lemma  \ref{Lem:hyper-semipiano}(4),
\[
V(\beta',0)=\{z\in \Ha: k_\Ha(z, (0,+\infty))<R_0\}.
\]
Making use of Riemann mappings, we conclude that if $\Delta\subsetneq \C$ is a simply connected domain and $\gamma:(0,+\infty)\to \Delta$ is a geodesic such that $\lim_{t\to 0^+}k_\Delta(\gamma(t), \gamma(1))=\lim_{t\to +\infty}k_\Delta(\gamma(t), \gamma(1))=+\infty$, then for every $t\in (0,+\infty)$,
\[
D^{hyp}_\Delta(\gamma(t), R_0)\subset \{z\in \Delta: k_\Ha(z, \gamma((0,+\infty)))<R_0\}.
\]
\end{remark}

We next present  two consequences of Lemma \ref{Lem:Stolz-qg}:

\begin{proposition}\label{Prop:suff-qc-cont}
Let $\Delta, U\subsetneq \C$ be two simply connected domains. Let $R>0$ and let  $\gamma:[0,+\infty)\to U$ be a geodesic in $U$ such that $\lim_{t\to+\infty}k_U(\gamma(t), \gamma(0))=+\infty$. Suppose
\[
S_U(\gamma, R)\subset\Delta\subseteq U.
\]
Then there exists $C>1$ such that for every $0\leq T<+\infty$ the curve $[0,T]\ni t\mapsto \gamma(t)$ is a $(C,0)$-quasi-geodesic in $\Delta$.

In particular, if $f:\D \to \Delta$ is a Riemann map, then $f^{-1}(\gamma(t))$ converges non-tangentially to a point $\sigma\in \partial \D$.
\end{proposition}
\begin{proof}
Fix $R_0\in (0, R)$. Note that $\gamma(t)\in S_U(\gamma, R_0)$ for all $t\in [0,+\infty)$. Hence, by  Lemma \ref{Lem:Stolz-qg}, and taking into account that $\gamma$ is a geodesic in $U$, for every $0\leq s\leq t<+\infty$, we have
\begin{equation*}
\begin{split}
\ell_\Delta(\gamma;[s,t])&=\int_s^t\kappa_{\Delta}(\gamma(u);\gamma'(u))du\leq \int_s^t\kappa_{S_U(\gamma, R)}(\gamma(u);\gamma'(u))du \\&\leq C\int_s^t\kappa_{U}(\gamma(u);\gamma'(u))du=C\ell_U(\gamma;[s,t])\\&=Ck_U(\gamma(s),\gamma(t))\leq C k_\Delta(\gamma(s), \gamma(t)),
\end{split}
\end{equation*}
which shows that $\gamma:[0, T]\to \Delta$ is a $(C,0)$-quasi-geodesic in $\Delta$ for all $T>0$.

Since $\Delta\subset U$, we have
\[
\lim_{t\to+\infty}k_\Delta(\gamma(0), \gamma(t))\geq \lim_{t\to+\infty}k_U(\gamma(0), \gamma(t))=+\infty.
\]
Hence, by Corollary \ref{Cor:shadow}, there exist a geodesic $\eta:[0,+\infty)\to \Delta$ for $\Delta$ and $\delta>0$ such that $\eta(0)=\gamma(0)$, $\lim_{t\to +\infty}k_\Delta(\eta(0), \eta(t))=+\infty$ and $k_\Delta(\gamma(s), \eta([0,+\infty)))<\delta$ for every $s\in [0,+\infty)$. Proposition \ref{Prop:conv-nt-sc} implies then the final statement.
\end{proof}

\begin{theorem}\label{Thm:nec-suff-non-tg}
Let $\Delta\subsetneq \C$ be a simply connected domain and let $f:\D\to \Delta$ be a Riemann map. Let $\{z_n\}\subset \Delta$ be a compactly divergent sequence. Then $\{f^{-1}(z_n)\}$ converges non-tangentially to a point $\sigma\in \partial \D$   if and only if there exist a simply connected domain $U\subsetneq \C$, a geodesic  $\gamma:[0,+\infty)\to U$ of $U$ such that $\lim_{t\to+\infty}k_U(\gamma(t), \gamma(0))=+\infty$ and $R>R_0>0$ such that
\begin{enumerate}
\item $S_U(\gamma, R)\subset\Delta\subseteq U$,
\item  there exists $n_0\geq 0$ such that $z_n\in S_U(\gamma, R_0)$   for all $n\geq n_0$.
\end{enumerate}
\end{theorem}
\begin{proof}
 If $\{f^{-1}(z_n)\}$ converges non-tangentially to a point in $\partial \D$, then the result follows trivially by taking $U=\Delta$ and appealing to Proposition \ref{Prop:conv-nt-sc}.

Conversely, since $S_U(\gamma, R)\subset \Delta\subset U$, by Proposition \ref{Prop:suff-qc-cont}, there exists $C>1$ such that the curve $[0,T]\ni r\mapsto \gamma(r)$ is a $(C,0)$-quasi-geodesic in $\Delta$ for all $T>0$ and, arguing as in the last part of the proof of Proposition \ref{Prop:suff-qc-cont}, we find a geodesic $\eta:[0,+\infty)\to \Delta$ for $\Delta$ and $\delta>0$ such that $\eta(0)=\gamma(0)$, $\lim_{t\to +\infty}k_\Delta(\eta(0), \eta(t))=+\infty$ and $k_\Delta(\gamma(s), \eta([0,+\infty)))<\delta$ for every $s\in [0,+\infty)$.

Fix $n\geq n_0$. By hypothesis, there exists $s_n\in [0,+\infty)$ such that $k_U(z_n, \gamma(s_n))<R_0$. Hence,  by Lemma \ref{Lem:Stolz-qg}
\[
k_\Delta(z_n, \gamma(s_n))\leq k_{S_U(\gamma, R)}(z_n, \gamma(s_n))\leq C k_U(z_n, \gamma(s_n))<CR_0.
\]
Let $u_n\in [0,+\infty)$ be such that $k_\Delta(\gamma(s_n), \eta(u_n))<\delta$. Then,
\[
k_\Delta(z_n, \eta(u_n))\leq k_\Delta(z_n, \gamma(s_n))+k_\Delta(\gamma(s_n), \eta(u_n))<CR_0+\delta.
\]
By the arbitrariness of $n$, this proves that for $n\geq n_0$, $z_n\in S_\Delta(\eta, CR_0+\delta)$. Proposition \ref{Prop:conv-nt-sc} implies then the statement.
\end{proof}

The previous results have practical applications. For instance, if $\Delta\subset \C$ is a simply connected domain such that $V(\beta, 0)\subset \Delta\subset \Ha$,  for some $\beta\in (-\pi/2, \pi/2)$, then the curve $(0,+\infty)\ni t\mapsto t$ is a quasi-geodesic in $\Delta$ and its pre-image via a Riemann map converges non-tangentially to the boundary.

When dealing with the slope problem for semigroups, it is useful to consider other simple domains, and we conclude this section with a corollary which will be useful later on.

\begin{definition}
The {\sl Koebe domain with base point $p\in \C$} is
\[
\mathcal K_p:=\C\setminus\{\zeta\in \C: \Re \zeta=\Re p, \Im\zeta\leq \Im p\}.
\]
\end{definition}

Since $\mathcal K_p$ is symmetric with respect to the line $\{\zeta\in \C: \Re \zeta=\Re p\}$, it follows that  the curve $\gamma_p:(0,+\infty)\ni t\mapsto p+it$ is a geodesic in $\mathcal K_p$. A simple direct computation, using the Riemann map $f:\zeta\mapsto \sqrt{-i\zeta}$  from $\mathcal K_0$ to $\Ha$, shows

\begin{lemma}\label{Lem:sector-koebe}
Let $p\in \C$ and let $R>0$. Fix $t_0>0$ and let $\gamma_p:[t_0,+\infty)\to \mathcal K_p$ be given by $\gamma_p(t)=p+it$, $t\geq t_0$. Then there exists $\beta\in (0,\pi)$ such that
\[
S_{\mathcal K_p}(\gamma_p, R)=\left((p+iV(\beta, 0))\setminus \{\zeta\in \C: |\zeta-p|\leq t_0\}\right)\cup D^{hyp}_{\mathcal K_p}(it_0+p, R).
\]
\end{lemma}

The following corollary gives a simple geometric condition for the preimage of a line to converge non-tangentially to the boundary of the disc:

\begin{corollary}\label{Cor:sector-implies-nt}
Let $\Delta\subsetneq \C$ be a simply connected domain and $f:\D \to \Delta$  a Riemann map. Suppose there exists  $p\in \C$
such that $\{p-it, t\geq 0\}\subset \C\setminus \Delta$ and $\{p+it, t> 0\}\subset \Delta$. If there exist $N\geq 0$ and $\beta\in (0,\pi)$ such that  $p+iN+iV(\beta,0)\subset \Delta$, then there exist $C>1$, $N'>0$, such that for every $T>N'$ the curve $[N', T]\ni t\mapsto p+it$ is a $(C,0)$-quasi-geodesic in $\Delta$. In particular, there exists
$\sigma\in \partial \D$ such that $(0,+\infty)\ni t\mapsto f^{-1}(p+it)$ converges non-tangentially to $\sigma$ as $t\to +\infty$.
\end{corollary}

\begin{proof} By hypothesis, $\Delta\subseteq \mathcal K_p$ and for every $a\geq 0$ the curve $\gamma_a:(a,+\infty)\to \mathcal K_p$, $\tilde\gamma(t)=p+it$ is a geodesic in $\mathcal K_p$. Therefore, according to Proposition \ref{Prop:suff-qc-cont}, it is enough to show that there exist $a\geq 0$ and $R>0$ such that $S_{\mathcal K_p}(\gamma_a, R)\subset \Delta$.

Let $\beta'\in (0,\beta)$. Let $w'$ be the point of intersection between $\{z=p+i\rho e^{i\beta'}, \rho>0\}$ and $\{z=p+iN+i\rho e^{i\beta}, \rho>0\}$. Let $R:=\inf_{\rho>0}k_{\mathcal K_p}(w', p+i\rho)$. By Remark \ref{Rem:disco-in-sector-qg}, for all $t>0$,
\begin{equation}\label{Eq:palla-koebe-insect}
D^{hyp}_{\mathcal K_p}(it+p, R)\subset \{z\in \mathcal K_p: k_{\mathcal K_p}(z, p+i(0,+\infty))<R\}=p+iV(\beta',0),
\end{equation}
 where the last equality follows at once using the biholomorphism $\mathcal K_p \ni \zeta\mapsto \sqrt{-i\zeta}\in \Ha$.

Let $t_0>N$. Since, $\lim_{t\to +\infty}k_{\mathcal K_p}(p+it_0, p+it)=+\infty$ and
\[
(p+iV(\beta',0)\cap \{z\in \C: |z-p|>|w'-p|\})\subset p+iN+iV(\beta, 0),
\]
 there exists $N'>0$ such that $D^{hyp}_{\mathcal K_p}(iN'+p, R)\subset p+iN+iV(\beta, 0)$.

Hence,
\begin{equation*}
\begin{split}
&\left((p+iV(\beta', 0))\setminus \{\zeta\in \C: |\zeta-p|\leq N'\}\right)\cup D^{hyp}_{\mathcal K_p}(p+iN', R)\\& \subset p+iN+iV(\beta,0) \subset \Delta,
\end{split}
\end{equation*}
which, by Lemma \ref{Lem:sector-koebe}, implies that $S_{\mathcal K_p}(\gamma_{N'}, R)\subset \Delta$, and we are done.
\end{proof}

As a corollary of the previous results we have:

\begin{proposition}\label{sector-implies-convergnt}
Let $(\phi_t)$ be a non-elliptic semigroup in $\D$, let $h$  be its Koenigs function, $\Omega=h(\D)$ and $\tau\in \partial \D$ its  Denjoy-Wolff point. If there exist $\beta\in (0,\pi/2)$ and $q\in \Omega$ such that $q+iV(\beta,0)\subset \Omega$, then $\phi_t(z)$ converges non-tangentially to $\tau$ for all $z\in \D$.
\end{proposition}

\begin{proof}
Let $p\in \partial \Omega$. Then $\Omega\subset \mathcal K_p$. Let $w_0$ be the point of intersection of the line $\{\zeta\in \C: \Re \zeta=\Re p\}$ and the boundary of the sector $q+iV(\beta,0)$. Since the domain is starlike at infinity, it follows at once that $w_0+iV(\beta,0)\subset \Omega$. Let $\alpha\in (0,\beta)$. Let $w'$ be the point of intersection between $\{z=p+i\rho e^{i\alpha}, \rho>0\}$ and $\{z=w_0+i\rho e^{i\beta}, \rho>0\}$. Let $R:=\inf_{\rho>0}k_{\mathcal K_p}(w', p+i\rho)$. Finally, since $D^{hyp}_{\mathcal K_p}(it+p, R)\subset V(\alpha, 0)$ for all $t>0$, there exists $t_0\in \R$ be such that $D^{hyp}_{\mathcal K_p}(it_0, R)\subset w_0+iV(\beta,0)$. Hence,
\begin{equation}\label{eq:insidekoe}
(p+iV(\alpha, 0)\setminus \{\zeta\in \C: |\zeta-p|\leq t_0\}\cup D^{hyp}_{\mathcal K_p}(p+it_0, R)\subset w_0+iV(\beta,0)\subset \Omega.
\end{equation}
Let $\gamma:[0,+\infty)$ be given by $\gamma(t)=p+i(t_0+t)$. The curve $\gamma$ is a geodesic for $\mathcal K_p$ such that $\lim_{t\to+\infty}k_{\mathcal K_p}(\gamma(0), \gamma(t))=+\infty$. By \eqref{eq:insidekoe} and Lemma \ref{Lem:sector-koebe}, $S_{\mathcal K_p}(\gamma, R)\subset \Omega$. Moreover, if $z_0\in \D$ is such that $h(z_0)=p+it_0$, then $h(z_0)+it\in S_{\mathcal K_p}(\gamma, R)$ for all $t\geq 0$. Hence, by Theorem \ref{Thm:nec-suff-non-tg}, $\phi_t(z_0)$ converges non-tangentially --- hence, $\phi_t(z)$ converges non-tangentially to $\tau$ for all $z\in \D$.
\end{proof}

As a direct corollary from the previous proposition we have

\begin{corollary}
Let $(\phi_t)$ be a non-elliptic semigroup in $\D$, let $h$  be its Koenigs function, $\Omega=h(\D)$ and $\tau\in \partial \D$ its  Denjoy-Wolff point. Suppose $w_0\in \Omega$. If
\[
\liminf_{t\to+\infty}\frac{\delta_\Omega(w_0+it)}{t}>0
\]
then $\phi_t(z)$ converges to $\tau$ non-tangentially, for all $z\in \D$.
\end{corollary}

\section{Good boxes and localization}\label{good}

The goal of this section is to prove that if a simply connected domain contains a rectangle whose  height is much larger than the base size---a ``good box''---then the hyperbolic geometry of the domain inside the rectangle is similar to that of a strip. We start with discussing hyperbolic geometry in the strip.

\begin{definition}
For $\rho>0$ we define the {\sl strip of width $\rho$} \index{Strip}
\[
\strip_\rho:=\{\zeta\in \C: 0<\Re \zeta<\rho\}.
\]
For $\rho=1$ we simply write $\strip:=\strip_1$. \end{definition}

\begin{proposition}\label{Prop:strip}
Let $a\in \R$ and $R>0$.
\begin{enumerate}
\item the curve $\gamma_0:\R\ni  t\mapsto a+\frac{R}{2}+it$ is a geodesic of $\strip_R+a$ and, for every $s<t$,
\[
k_{\strip_R+a}(a+\frac{R}{2}+is,a+\frac{R}{2}+it)=\frac{\pi (t-s)}{2R}.
\]
\item For every $z\in \strip_R+a$, the orthogonal projection $\pi_{\gamma_0}(z)$ of $z$ onto $\gamma_0$, {\sl i.e.}, the (only) point $\pi_{\gamma_0}(z)$ such that $k_{\strip_R+a}(z,\gamma_0)= k_{\strip_R+a}(z,\pi_{\gamma_0}(z))$, is
\[
\pi_{\gamma_0}(z)=a+\frac{R}{2}+i\Im z.
\]
\item For every $y\in \R$, the curve $(-\frac{R}{2},\frac{R}{2})\ni s\mapsto s+a+\frac{R}{2}+iy$ is a geodesic of $\strip_R+a$ and for all $s_1, s_2\in (-\frac{R}{2},\frac{R}{2})$, $0<s_2<s_1$ or $s_1<s_2<0$,
\[
\frac{1}{2}\log \frac{R-2|s_2|}{R-2|s_1|}\leq k_{\strip_R+a}(s_1+a+\frac{R}{2}+iy, s_2+a+\frac{R}{2}+iy)\leq \log \frac{R-2|s_2|}{R-2|s_1|}.
\]
\item For every $\delta>0$, the hyperbolic sector $S_{\strip_R+a}(\gamma_0, \delta)=\strip_r+a+\frac{R-r}{2}$, for some $r<R$. Moreover, if $z\in S_{\strip_R+a}(\gamma_0, \delta)$,  then $|\Re z-a-\frac{R}{2}|<\frac{R}{2}(1-e^{-2\delta})$. While, if $z\not\in S_{\strip_R+a}(\gamma_0, \delta)$, then, setting $u(z)=\hbox{sgn}(\Re z-a-\frac{R}{2})$,
\[
k_{\strip_R+a}(z,S_{\strip_R+a}(\gamma_0, \delta))=k_{\strip_R+a}(z, a+\frac{R}{2}+u(z) r+i\Im z).
\]
\item For every $M_2,M_1\in \R$, $M_2>M_1$, let
\[
Q(M_1,M_2):=\inf \{k_{\strip_R+a}(z,w): z,w\in \strip_R+a,\Im z\leq M_1, \Im w\geq M_2\}.
\]
 Then
\[
Q(M_1,M_2)
=\frac{\pi(M_2-M_1)}{2R}.
\]
\item For every $\delta>0$ and $N_0>0$ there exists $N>N_0$  which does not depend on  $R, a$, such that for every $M_1, M_2\in \R$ with $M_2-M_1> RN$, there exists $q\in (M_1, M_2-RN_0)$, such that
every geodesic  $\gamma$ in $\strip_R+a$ joining two points $z, w\in \strip_R+a$ with $\Im w>M_2$ and $\Im z<M_1$ satisfies $\gamma\cap \{\zeta\in \C: q<\Im \zeta< q+RN_0\}\subset  S_{\strip_R+a}(\gamma_0, \delta)$.
\item For every $z, w\in \strip_R+a$ with $\Im w\geq \Im z$, the geodesic joining $z$ and $w$ is contained in $\{\zeta\in \strip_R+a: \Im z\leq \Im \zeta\leq \Im w\}$.
\end{enumerate}
\end{proposition}
\begin{proof}
The holomorphic function $f:\Ha\ni z\mapsto \frac{Ri}{\pi}\log z+\frac{R}{2}+a$ is a biholomorphism from $\Ha$ to $\strip_R+a$.

(1) Since $\strip_R+a$ is symmetric with respect to the line $\{z\in \C: \Re z=a+\frac{R}{2}\}$, it follows  that $\gamma_0$ is a geodesic. The formula for the hyperbolic distance  follows at once by a direct computation using $f$ and the corresponding expression of $k_\Ha$.

(2) Using  the biholomorphism $f$, this amounts to proving that for every $\rho_1>\rho_2>0$ and $\theta_1, \theta_2\in (-\pi/2,\pi/2)$ we have
\[
k_{\Ha}(\rho_1 e^{i\theta_1}, \rho_2 e^{i\theta_2})\geq k_\Ha(\rho_1,\rho_2),
\]
which follows directly from Lemma \ref{Lem:hyper-semipiano}(3).

(3) By symmetry,  for every $y\in \R$, the curve $\eta:(-\frac{R}{2},\frac{R}{2})\ni s\mapsto s+a+\frac{R}{2}+iy$ is a geodesic. Let $-\frac{R}{2}<s_1<s_2<\frac{R}{2}$. By Theorem \ref{Thm:Distance-Lemma},
\[
\frac{1}{2}\int_{s_1}^{s_2}\frac{|\eta'(s)|}{\delta_{\strip_R+a}(\eta(s))}ds\leq k_{\strip_R+a}(\eta(s_1), \eta(s_2))\leq \int_{s_1}^{s_2}\frac{|\eta'(s)|}{\delta_{\strip_R+a}(\eta(s))}ds.
\]
Simple geometric consideration shows that $\delta_{\strip_R+a}(\eta(s))=\frac{R}{2}-|s|$.  Hence, the estimates follow from a direct computation.

(4) Using again the biholomorphism $f$, it is easy to see that   $\gamma_0$ corresponds to the geodesic $(0,+\infty)$ in $\Ha$ and by Lemma \ref{Lem:hyper-sector-inH},
\[
f^{-1}(S_{\strip_R+a}(\gamma_0, \delta))=S_\Ha((0,+\infty), \delta)=V(\beta,0)
\]
 for some $\beta\in (0,\pi/2)$. Hence,
$S_{\strip_R+a}(\gamma_0, \delta)=\strip_r+a+\frac{R-r}{2}$, for some $r<R$.

Next, assume $z\in S_{\strip_R+a}(\gamma_0, \delta)$ and let $s_1:=\Re z-a-\frac{R}{2}$. Hence, by (2),
\begin{equation*}
\begin{split}
\delta&>k_{\strip_R+a}(z, \gamma_0)=k_{\strip_R+a}(z, a+\frac{R}{2}+i\Im z)\\&=k_{\strip_R+a}(s_1+a+\frac{R}{2}+i\Im z, a+\frac{R}{2}+i\Im z),
\end{split}
\end{equation*}
and from the lower estimate in (3) we obtain
\[
\frac{1}{2}\log \frac{R}{R-2|s_1|}<\delta.
\]
A direct computation shows that this is equivalent to $|\Re z-a-\frac{R}{2}|<\frac{R}{2}(1-e^{-2\delta})$.

Finally, if $z\not\in S_{\strip_R+a}(\gamma_0, \delta)$, using again $f$, the problem reduces to show that, given $\rho e^{i\theta}=f^{-1}(z)$, with $\rho>0$ and $\theta\in (\beta, \pi/2)$, (the case $\theta\in (-\pi/2,\beta)$ is analogous), then
\[
k_\Ha(\rho e^{i\theta}, V(\beta,0))=k_\Ha(\rho e^{i\theta}, \rho e^{i\beta}).
\]
This follows at once from Lemma \ref{Lem:hyper-semipiano}(3).

(5) It is clear that $Q(M_1,M_2)=\inf \{k_{\strip_R+a}(z,w): z,w\in \strip_R+a,\Im z= M_1, \Im w= M_2\}$. Using the biholomorphism $f$, we see that  $\{\zeta\in \strip_R+a: \Im \zeta=M_j\}$ is mapped onto $\{\rho_j e^{i\theta}: \theta\in (-\pi/2,\pi/2)\}$ for some $0<\rho_1<\rho_2$. Hence, the statement is equivalent to
\[
\inf_{\theta, \tilde\theta\in (-\pi/2,\pi/2)}k_{\Ha}(\rho_1 e^{i\theta},  \rho_2 e^{i\tilde\theta})\geq k_\Ha(\rho_1,  \rho_2),
\]
which follows from  Lemma \ref{Lem:hyper-semipiano}(5).

(6) Fix $\delta>0, N_0>0$. We already saw that $f^{-1}(S_{\strip_R+a}(\gamma_0, \delta))=V(\beta,0)$ for some $\beta\in (0,\pi/2)$.

Now (see Figure \ref{fig:slop}),
\begin{figure}[h]
\centering 
\begin{tikzpicture}[scale = 0.8]
\draw [dashed] (0,-7.5) -- (0,7.5) ;
\draw [dashed] (0,0) -- (7,0) ;

\draw (0,-6) arc (-90:90:6) ;
\draw (0,-1) arc (-90:90:1) ;
\draw (0,1) arc (-90:90:2.5) ;
\draw (0,-6) arc (-90:90:2.5) ;
\draw (0,-4.5) arc (-90:90:4.5) ;

\draw (0,0) -- (4.975,7.5) ;
\draw (4.7,7) node[scale=1][right]{$L^+$};

\draw (0,0) -- (4.975,-7.5) ;
\draw (4.7,-7) node[scale=1][right]{$L^-$};

\draw (0,1) arc (-90:90:1.18) ;
\draw (0,-1) arc (90:-90:1.18) ;
\draw (0,-1.33) arc (-90:90:1.33) ;

\draw (2.487,3.75) node[scale=1]{$\bullet$} ;
\draw (2.487,3.8) node[scale=1][right]{$q_2^+$};

\draw (2.487,-3.75) node[scale=1]{$\bullet$} ;
\draw (2.487,-3.75) node[scale=1][right]{$q_2^-$};

\draw (0.731,1.101) node[scale=1]{$\bullet$} ;
\draw (0.77,1.101) node[scale=1][right]{$q_1^+$};

\draw (0.731,-1.101) node[scale=1]{$\bullet$} ;
\draw (0.78,-1.01) node[scale=1][right]{$q_1^-$};

\draw (0,1) node[scale=1][left]{$i$};
\draw (0,-1) node[scale=1][left]{$-i$};
\draw (0,6) node[scale=1][left]{$pi$};
\draw (0,-6) node[scale=1][left]{$-pi$};

\draw (0,0.5) node[scale=1][left]{$A^-$};
\draw[->,>=latex] (-0.2,0.4) to (0.5,0.2) ;

\draw (6.9,3) node[scale=1][left]{$A^+$};

\draw (4,-0.65) node[scale=1][left]{$Q$};

\draw (-0.2,2.5) node[scale=1][left]{$C_0^+$};
\draw[->,>=latex] (-0.25,2.5) to (0.9,2.9) ;

\draw (-0.2,5) node[scale=1][left]{$F_0^+$};
\draw[->,>=latex] (-0.25,5) to (1.55,5.45) ;

\fill [color=gray!20, pattern=north east lines] (0,-1) arc (-90:90:1) ;

\draw [dashed] (3.3,-3) -- (1.8,-1.5) ;
\draw [dashed] (3.1,-2.1) -- (1.6,-0.6) ;
\draw [dashed] (2.8,-1.1) -- (1.3,0.4) ;
\draw [dashed] (4.1,-1.6) -- (2.6,-0.2) ;
\draw [dashed] (3.3,-0.3) -- (1.8,1.2) ;
\draw [dashed] (4.1,-0.3) -- (2.6,1.2) ;
\draw [dashed] (3.7,0.7) -- (2.2,2.2) ;
\draw [dashed] (3.9,1.2) -- (2.4,2.7) ;
\draw [dashed] (4.1,1.7) -- (2.6,3.2) ;


\draw [dashed] (0.8,6.2) -- (2.5,7.9) ;
\draw [dashed] (2.5,6) -- (4.2,7.7) ;
\draw [dashed] (4.65,4.9) -- (6.35,6.6) ;
\draw [dashed] (5.2,3.2) -- (6.9,4.9) ;
\draw [dashed] (6,1.7) -- (7,2.7) ;
\draw [dashed] (6.2,0.5) -- (7,1.3) ;

\draw [dashed] (0.6,-7.7) -- (2.3,-6) ;
\draw [dashed] (2.1,-7.7) -- (3.8,-6) ;
\draw [dashed] (4.65,-6.6) -- (6.35,-4.9) ;
\draw [dashed] (5,-5.1) -- (6.9,-3.2) ;
\draw [dashed] (6,-2.7) -- (7,-1.7) ;
\draw [dashed] (6.2,-1.3) -- (7,-0.5) ;


\end{tikzpicture}
\caption{}\label{fig:slop}
\end{figure}
 let $C_0$ be the  circle with center $\frac{i}{1-\cos \beta}$ and radius $\frac{\cos \beta}{1-\cos\beta}$ and let $C_0^+=C_0\cap \Ha$. Note that, since the center of $C_0$ is on the imaginary axis, $C_0$ intersects orthogonally $i\R$. Hence, $C_0^+$ is a geodesic in $\Ha$. Moreover,  it is easy to see that for $x>0$, the Euclidean distance from $ix$ to $L^+:=\{\rho e^{i\beta}:\rho>0\}$ is $x\cos\beta$, so that $C_0$ is tangent to $L^+$. Also, the end points  of $C_0^+$ are $i$ and $\frac{1+\cos\beta}{1-\cos \beta} i$.

Let now $F_0^+=F_0\cap \Ha$, where $F_0$ is the circle orthogonal to $i\R$ and passing through $i$ and $p i$, for some $p>\frac{1+\cos \beta}{1-\cos\beta}$ to be chosen later. Note that by construction $F_0^+$ intersects $L^+$ into two points, $q_1^+$ and $q_2^+$, $|q_1^+|<|q_2^+|$.
Let $F_0^-$ be the reflection of $F_0^+$ about the real axis, that is, $F_0^-$ is the circle orthogonal to $i\R$ passing through $-i$ and $-p i$.

Let $U^\pm$ be the unbounded connected component of $\Ha\setminus F_0^\pm$. By Lemma \ref{Lem:total-geo-disc}, $U^+$, $U^-$ are totally geodesic. Let $U:=U^+\cap U^-$. Then $U$ is totally geodesic as well since for every two points of $U$ the geodesic joining them is contained in $U^+$ and $U^-$.

Let  $A^-:=\{\rho e^{i\theta}: 0<\rho <1, |\theta|<\pi/2\}$, $A^+:=\{\rho e^{i\theta}: \rho>p, |\theta|<\pi/2\}$,  $\tilde Q:=\{\rho e^{i\theta}: |q_1^+|<\rho<|q_2^+|, |\theta|<\pi/2\}$ and $Q=\tilde Q\cap U$.

Note that $A^-, A^+\subset U$, hence, if $\zeta_0, \zeta_1\in \Ha$ are such that $\zeta_0\in A^-$ and $\zeta_1\in A^+$, the geodesic $\eta:[0,1]\to \Ha$ of $\Ha$ joining $\zeta_0$ and $\zeta_1$ is contained in $U$ and, by construction, it necessarily crosses $V(\beta,0)$. Moreover, by construction, for all $t\in (0,1)$ such that  $\eta(t)\in \tilde Q$, the point $\eta(t)\in Q$.

Since $|q_2^+|/|q_1^+|\to 1$ for $p\to \frac{1+\cos \beta}{1-\cos\beta}$ and $|q_2^+|/|q_1^+|\to +\infty$ for $p\to+\infty$, given $N_0>0$ we can find $p$ such that $\frac{1}{\pi}[\log |q_2^+|-\log |q_1^+|]= N_0$. Let   $p$ be such a number and let $N:=\frac{1}{\pi}\log p$.

Note that $N$  depends only on $\beta$---hence, on $\delta$ and on $N_0$ and that $N>N_0$. A simple computation shows that $f(A^-)=\{z\in \strip_R+a: \Im z<0\}$, and $f(A^+)=\{z\in \strip_R+a: \Im z>\frac{R}{\pi}\log p \}$. Moreover, $f^{-1}(\tilde Q)=\{z\in \strip_R+a: \frac{R\log |q_1^+|}{\pi}<\Im z<\frac{R\log |q_2^+|}{\pi}\}$.

 Therefore, since $f$ maps geodesics of $\Ha$ onto geodesics of $\strip_R+a$, the previous argument shows that for every $z\in \{z\in \strip_R+a: \Im z<0\}$ and $w\in \{z\in \strip_R+a: \Im z>\frac{R}{\pi}\log p \}$ the geodesic $\gamma$ joining $z$ and $w$ satisfies
 \[
 \gamma \cap \{z\in \strip_R+a: \frac{R\log |q_1^+|}{\pi}<\Im z<\frac{R\log |q_2^+|}{\pi}\}\subset S_{\strip_R+a}(\gamma_0, \delta).
 \]

Finally, given $M_1, M_2\in \R$ such that $M_2-M_1>RN$, one can reduce to the previous case using  automorphisms of $\strip_R+a$ of the form $z\mapsto z-ik$, $k\in \R$, and taking into account that such automorphisms are isometries for $k_{\strip_R+a}$ and map $\gamma_0$ onto $\gamma_0$.

(7) If $\Im z=\Im w$, the result follows from (2). If $\Im w>\Im z$, we saw in (6) that for all $\epsilon>0$, $\{\zeta\in \strip_R+a: \Im \zeta< \Im w+\epsilon\}$ and   $\{\zeta\in \strip_R+a: \Im \zeta> \Im z+\epsilon\}$ are totally geodesic in $\strip_R+a$. Hence, their intersection is. Therefore, for all $\epsilon>0$ the geodesic of $\strip_R+a$ joining $z$ and $w$ is contained in $\{\zeta\in \strip_R+a: \Im z+\epsilon<\Im \zeta< \Im w+\epsilon\}$. By the arbitrariness of $\epsilon$, we get the result.
\end{proof}

We now present  several  localization results which will be useful in the subsequent constructions.

Let $a, b\in \R$ and $R>0$. Let
\[
\Omega_{a,b,R}:=\C\setminus \{z\in \C: \Re z\in \{a, a+R\}, \Im z\leq b\}.
\]

\begin{proposition}\label{Prop:local-strip-bound}
Let  $c>1$. Then there exists $D(c)>0$ with the following properties.  Let $D\geq D(c)$ and $R>0$, $a,b\in \R$. Then for all $v\in \C$ and $z\in (\strip_R+a)$ such that $\Im z\leq b-RD$,
\[
\kappa_{\Omega_{a,b,R}}(z;v)\leq \kappa_{\strip_R+a}(z;v)\leq c \kappa_{\Omega_{a,b,R}}(z;v).
\]
 Moreover,  for every $z, w\in (\strip_R+a)$ such that $\Im z, \Im w\leq b-RD$
\[
k_{\Omega_{a,b,R}}(z,w)\leq k_{\strip_R+a}(z,w)\leq c k_{\Omega_{a,b,R}}(z, w).
\]
\end{proposition}
\begin{proof}
The inequalities on the left hand side follow immediately  since $\strip_R+a\subset \Omega_{a,b,R}$.

Assume now $R=1, a=b=0$ and let $\Omega:=\Omega_{0,0,1}$.

For $n\in \N$ let $C_n:=\{\zeta\in \C: 0\leq \Re \zeta\leq 1, \Im \zeta=-n\}$. Clearly, $(C_n)$ is a null chain in $\Omega$, which represents a prime end $\underline{x}$ of $\Omega$.  Let $\strip^\ast$ be the open set in $\widehat{\Omega}$ defined by $\strip$ (that is, $\strip^\ast$ is the union of $\strip$ and all prime ends for which a representing null chain is eventually contained in $\strip$). Hence, $\strip^\ast$ is an open neighborhood of $\underline{x}$, since, by construction, the interior part  of $C_n$ belongs to $\strip$ for all $n\geq 1$. Moreover, $\strip^\ast\cap \Omega=\strip$, which is simply connected. Therefore, we can apply Theorem \ref{Thm:localiz} to $\underline{x}$ and $\strip^\ast$ and come up with an open set $V^\ast\subset \strip^\ast$ in $\widehat{\Omega}$ which contains $\underline{x}$ and such that
\begin{equation}\label{Eq:estima-banda-out}
\kappa_{\strip}(z;v)\leq c \kappa_{\Omega}(z;v), \quad
 k_{\strip}(z,w)\leq c k_{\Omega}(z, w),
\end{equation}
for all $z, w \in V:=V^\ast\cap \Omega$ and $v\in \C$. Note that since $V^\ast$ is an open neighborhood of $\underline{x}$,  there exists $n_0\in \N$ such that the interior part of $C_n$ is contained in $V$ for all $n\geq n_0$. In particular, \eqref{Eq:estima-banda-out} holds for every $\zeta\in \strip$ such that $\Im \zeta\leq -(n_0+1)$.

Hence, we have proved the result with $D:=-(n_0+1)$ for $R=1, a=b=0$.

Now, assume $R>0$ and $a, b\in \R$. Using the map $\C\ni z\mapsto \frac{1}{R}(z-a-ib)\in \C$, which is  a biholomorphism from $\Omega_{a,b,R}$ to $\Omega$ and maps $(\strip_R+a)$ onto $\strip$ and $\{\zeta\in \strip_R+a: \Im \zeta\leq b-RD\}$ onto $\{\zeta\in \strip: \Im \zeta\leq -D\}$, the result follows at once from \eqref{Eq:estima-banda-out}.
\end{proof}

The next localization result  is a sort of converse of the previous one: we choose the part we want to localize and come up with a constant for the localization.
 We start with a definition:

\begin{definition}
Let $M\in \R, R>0$. The {\sl semi-strip}  of {\sl width $R$ and height $M$}  is
\[
\strip^M_R:=\{\zeta\in \C: 0<\Re \zeta<R, \Im \zeta>M\}.
\]
\end{definition}

\begin{proposition}\label{Prop:estim-strip-var2}
For every $E>0$ there exists $c'=c'(E)>1$ with the following properties. Let $a\in \R$,  $M\in \R$ and $R>0$. Then for all $v\in \C$ and $z\in (\strip_R+a)$ such that $\Im z\geq RE+M$,
\[
\kappa_{\strip_R+a}(z;v)\leq \kappa_{\strip^M_R+a}(z;v)\leq c'\kappa_{\strip_R+a}(z;v).
\]
Moreover, for every $z, w\in (\strip_R+a)$ such that $\Im z, \Im w> RE+M$,
\[
k_{\strip_R+a}(z,w)\leq k_{\strip^M_R+a}(z,w)\leq c' k_{\strip_R+a}(z,w).
\]
\end{proposition}
\begin{proof}
The left-hand side estimates follow immediately since $\strip^M_R+a\subset \strip_R+a$.

In order to prove the right-hand side estimates, arguing as in Proposition  \ref{Prop:local-strip-bound}, it is enough to prove the result for $R=1$, $a=0$, $M=0$ and then use the affine map $z\mapsto \frac{1}{R}(z-a-iM)$ to pass to the general case.

Fix $E>0$. Let $K:=\{z\in \C: E\leq \Re z\leq 1-E, E\leq \Im z\leq 1\}$ (possibly $K$ is empty if $E>1$). For $z\in \strip_1^0$ such that $\Im z\geq E$ and $z\not\in K$, we have $\delta_{\strip_1^0}(z)=\delta_{\strip_1}(z)$, hence, from Theorem \ref{Thm:Distance-Lemma-inf},
\[
\kappa_{\strip_1^0}(z;v)\leq \frac{|v|}{\delta_{\strip_1^0}(z)}=2 \frac{|v|}{2\delta_{\strip_1}(z)}\leq 2\kappa_{\strip_1}(z;v).
\]
In case $K$ is non-empty,  $K$ is compact in $\strip_1^0$ and in $\strip_1$. Since the hyperbolic metric is continuous in $z$, the following numbers are well defined:
\[
q:=\min_{z\in K}\kappa_{\strip_1}(z;1), \quad Q:= \max_{z\in K}\kappa_{\strip^0_1}(z;1).
\]
Moreover, $q>0$ (for otherwise the hyperbolic norm of $1$ would be $0$ at an interior point). Hence, for $z\in K$ and $v\in \C$,
\[
\kappa_{\strip_1^0}(z;v)=|v|\kappa_{\strip_1^0}(z;1)\leq \frac{|v|Q}{q}q\leq \frac{|v|Q}{q}\kappa_{\strip_1}(z;1)=\frac{Q}{q}\kappa_{\strip_1}(z;v).
\]
Taking $c'=\max\{2, \frac{Q}{q}\}$ we have the first estimate.

In order to prove the second inequality, note that $(0,1)+iE$ is  a geodesic in $\strip_1$ by symmetry. Then, Lemma \ref{Lem:total-geo-disc} guarantees that $\strip_1^E$ is totally geodesic in $\strip_1$. Therefore, given $z, w\in \strip_1^E$, let $\gamma:[0,1]\to \strip_1$ be the geodesic for $\strip_1$ which joins $z$ and $w$. Hence, $\gamma([0,1])\subset \strip_1^E$. Therefore, for what we have already  proved,
\begin{equation*}
\begin{split}
k_{\strip^0_1}(z,w)&\leq \ell_{\strip^0_1}(\gamma;[0,1])=\int_0^1\kappa_{\strip^0_1}(\gamma(t);\gamma'(t))dt \\&\leq c' \int_0^1\kappa_{\strip_1}(\gamma(t);\gamma'(t))dt = c' k_{\strip_1}(z,w),
\end{split}
\end{equation*}
and we are done.
\end{proof}

The next result allows us to estimate the hyperbolic distance and the displacement of geodesics in simply connected domains which contain ``good boxes'':

\begin{proposition}\label{good box}
Let $c>1$, let $D\geq D(c)$, where $D(c)>0$ is  given by Proposition \ref{Prop:local-strip-bound}, and fix $E\in(0, D)$. Then there exist $\epsilon=\epsilon(c,D,E)>0$ and $C=C(c,D,E)>1$ with the following property. Let $\Omega\subset \C$  be any simply connected domain such that
\begin{enumerate}
\item $\Omega\subset \Omega_{a,b,R}$ for some   $a, b\in \R$ and $R>0$
\item $\strip^M_R+a\subset \Omega$ for some $-\infty\leq M<b$,
\item $b-M>2RD$.
\end{enumerate}
Let
\begin{equation}\label{Eq:good-box-def}
B:=\{\zeta\in (\strip_R+a): M+RE < \Im \zeta < b-RD\}.
\end{equation}
Then, if $\gamma:[u_0,u_1]\to \Omega$ is a geodesic for $\Omega$ contained in $B$,  and  $\eta:[v_0,v_1]\to \strip_R+a$ is the geodesic in $\strip_R+a$ such that  $\gamma(u_j)=\eta(v_j)$, $j=0,1$, then for every $u\in [u_0,u_1]$ and $v\in [v_0,v_1]$,
\begin{equation}\label{Eq:box-strip-Om}
k_{\strip_R+a}(\gamma(u), \eta)< \epsilon, \quad k_{\strip_R+a}(\eta(v), \gamma)< \epsilon.
\end{equation}
Moreover, if $\eta:[v_0,v_1]\to B$ is a geodesic for $\strip_{R}+a$ and $\gamma:[u_0,u_1]\to \Omega$ is the geodesic for $\Omega$ such that  $\gamma(u_j)=\eta(v_j)$, $j=0,1$,  then for every $u\in [u_0,u_1]$ and $v\in [v_0,v_1]$,
\begin{equation}\label{Eq:box-Om-strip}
k_{\Omega}(\gamma(u), \eta)< \epsilon, \quad k_{\Omega}(\eta(v), \gamma)< \epsilon.
\end{equation}
In addition,  for every $z,w \in B$,
\[
\frac{1}{C} k_{\strip_R+a}(z,w)\leq k_\Omega(z,w)\leq C k_{\strip_R+a}(z,w).
\]
\end{proposition}
\begin{proof} Let $\Omega\subset\C$ be a simply connected domain which satisfies condition (1)---(3). Condition (3) implies that $M+RE<M+RD<b-RD$, therefore, $B\neq\emptyset$. Let  $\gamma:[u_0,u_1]\to B$ be a geodesic for $\Omega$.  Then for every $u_0\leq s\leq t\leq u_1$,  by Proposition \ref{Prop:local-strip-bound}  and by (1),
\begin{equation*}
\begin{split}
\ell_{\strip_R+a}(\gamma;[s,t])&=\int_s^t \kappa_{\strip_R+a}(\gamma(r);\gamma'(r))dr \leq c \int_s^t \kappa_{\Omega_{a,b,R}}(\gamma(r);\gamma'(r))dr \\&\leq c \int_s^t \kappa_{\Omega}(\gamma(r);\gamma'(r))dt=c k_\Omega(\gamma(s), \gamma(t)).
\end{split}
\end{equation*}
On the other hand, by (2)  and Proposition \ref{Prop:estim-strip-var2},
\[
k_\Omega(\gamma(s), \gamma(t))\leq k_{\strip_R^M+a}(\gamma(s), \gamma(t))\leq c' k_{\strip_R+a}(\gamma(s), \gamma(t)).
\]
Hence, $\ell_{\strip_R+a}(\gamma;[s,t])\leq cc' k_{\strip_R+a}(\gamma(s), \gamma(t))$. This proves that every geodesic $\gamma$ for $\Omega$ which is contained in $B$ is a $(cc',0)$-quasi-geodesic in $\strip_R+a$. Hence, by Theorem~\ref{Gromov}, there exists $\epsilon=\epsilon(c,c')=\epsilon(c,D,E)>0$ such that \eqref{Eq:box-strip-Om} holds.

In order to prove \eqref{Eq:box-Om-strip},  we argue similarly. Given $\eta:[v_0,v_1]\to B$  a geodesic for $\strip_{R}+a$, for all $v_0<s<t<v_1$, using  Proposition  \ref{Prop:local-strip-bound} and Proposition \ref{Prop:estim-strip-var2}, we have
\begin{equation*}
\begin{split}
\ell_{\Omega}(\eta;[s,t])&=\int_s^t \kappa_{\Omega}(\eta(r);\eta'(r))dr \leq  \int_s^t \kappa_{\strip_R^M+a}(\eta(r);\eta'(r))dr \\&\leq c' \int_s^t \kappa_{\strip_R+a}(\eta(r);\eta'(r))dt=c' k_{\strip_R+a}(\eta(s), \eta(t))\\
&\leq c'ck_{\Omega_{a,b,R}}(\eta(s),\eta(t))\leq c'ck_{\Omega}(\eta(s),\eta(t)).
\end{split}
\end{equation*}
Hence $\eta$ is a $(cc',0)$-quasi-geodesic in $\Omega$. Theorem \ref{Gromov} implies \eqref{Eq:box-Om-strip} with the same $\epsilon$ as before.

In order to prove the last inequalities, we note that $\strip^M_R+a\subset \Omega\subset \Omega_{a,b,R}$, hence,  for every $z,w\in B$,
\[
k_{ \Omega_{a,b,R}}(z,w)\leq k_\Omega(z,w)\leq k_{\strip^M_R+a}(z,w).
\]
Therefore, the result follows at once from Proposition  \ref{Prop:local-strip-bound} and Proposition \ref{Prop:estim-strip-var2} by taking $C=\max\{c,c'\}$.
\end{proof}

\begin{definition}
The set $B$ defined in \eqref{Eq:good-box-def} is a {\sl good box for $\Omega$ for the data $(c,D,E)$}\index{Good box for a simply connected domain}. Its {\sl width}\index{Width of a good box} is $R$ and its {\sl height}\index{Height of a good box} is $b-M-R(D+E)$. The segment $\{z=a+\frac{R}{2}+it, M+RE < t < b-RD\}$ is called the {\sl vertical bisectrix}\index{Vertical bisectrix of a good box} of $B$ and we denote it by $\hbox{bis}(B)$.
\end{definition}

With the help of the previous results, we prove now that ``long'' geodesics for $\Omega$ in a good box get close to the vertical bisectrix of the good box in a controlled way.

\begin{corollary}\label{cor:good-box-estim}
Let $c>1$ and let $D\geq D(c)$, where $D(c)>0$ is  given by Proposition \ref{Prop:local-strip-bound} and fix $E\in(0, D)$.   Let $\epsilon>0$  and $C>1$ be given by Proposition \ref{good box} and let $N_0> \frac{4(C+1)\epsilon}{\pi}$.
Let $\delta>0$ and let $N>N_0>0$ be the constant given by Proposition \ref{Prop:strip}(6). Finally,  let  $N_1:=N_0-\frac{4\epsilon}{\pi}$ and $N_2:=\frac{N_0}{2}-\frac{ 2(C+1)\epsilon}{\pi}$.

Let $\Omega\subset \C$  be a simply connected domain and assume $B\subset \Omega$ is a good box of $\Omega$ for the data $(c,D,E)$ of width $R>0$, height $h>0$ and vertical bisectrix $\hbox{bis}(B)=\{z=a+\frac{R}{2}+it, r_0 < t < r_0+h\}$, where $a, r_0\in \R$. Suppose $h>NR$. Let $I\subset \R$ be an open interval and let $\gamma: I\to \Omega$ be a geodesic of $\Omega$. Suppose there exists an interval
$[u_0,u_1]\subset I$ such that $\gamma([u_0,u_1])\subset B$ and
 $\Im \gamma(u_1)-\Im \gamma(u_0)>NR$. Then, there exists $r_1\in (r_0, r_0+h-R(N_1+\frac{2\epsilon}{\pi}))$  such that for every $t\in [u_0,u_1]$ for which $r_1<\Im \gamma(t)<r_1+RN_1$,
\begin{equation}\label{Eq:metr-2}
\gamma(t)\in \{z\in B: |\Re z-a-\frac{R}{2}|< \frac{R}{2}(1-e^{-2(\epsilon+\delta)})\}.
\end{equation}
Moreover, let $u_0':=\min \{t\in [u_0,u_1]: \Im \gamma(t)\geq r_1\}$ and $u_1':=\max \{t\in [u_0,u_1]: \Im \gamma(t)\leq r_1+RN_1\}$. Then for every $t\in I\setminus [u'_0,u'_1]$ it follows that $\gamma(t)\not\in \{z\in B: r_1+(\frac{N_1}{2}-N_2)R<\Im z<r_1+(\frac{N_1}{2}+N_2)R\}$.
\end{corollary}
\begin{proof}
Let $\gamma: I\to \Omega$ be a geodesic of $\Omega$. Suppose there exists an interval
$[u_0,u_1]\subset I$ such that $\gamma([u_0,u_1])\subset B$ and
 $\Im \gamma(u_1)-\Im \gamma(u_0)>NR$. Let $\eta:[v_0,v_1]\to \strip_R+a$ be the geodesic of $\strip_R+a$ such that $\eta(v_j)=\gamma(u_j)$, $j=0,1$.

Note that $\eta([v_0,v_1])\subset B$ by Proposition \ref{Prop:strip}(7). Moreover, since $\Im \eta(v_1)-\Im \eta(v_0)>NR$, by Proposition \ref{Prop:strip}(6), there exists  $q\in (\Im \eta(v_0),\Im \eta(v_1)-RN_0)$ such that $\eta(v)\in   S_{\strip_R+a}(\gamma_0, \delta)$ for all $v\in [v_0,v_1]$ for which $q<\Im \eta(v)<q+RN_0$. By Proposition \ref{Prop:strip}(4), in fact, $k_{\strip_R+a} (a+\frac{R}{2}+i\Im \eta(v), \eta(v))=k_{\strip_R+a}(\gamma_0, \eta(v))<\delta$.

Let
\[
r_1:=q+\frac{2\epsilon R}{\pi}.
\]
Note that
\[
r_0<q<r_1=q+\frac{2\epsilon R}{\pi}<r_0+h-RN_0+\frac{2\epsilon R}{\pi}=r_0+h-R(N_1+\frac{2\epsilon}{\pi}).
\]
 and $r_1+RN_1<q+RN_0$. Let
\[
B^G:=\{z\in \strip_R+a: r_1<\Im z<r_1+RN_1\}.
\]
By Proposition \ref{good box},  for every $u\in [u_0,u_1]$,
there exists $v_u\in [v_0,v_1]$ such that $k_{\strip_R+a}(\eta(v_u), \gamma(u))< \epsilon$.

Let $u\in [u_0,u_1]$ be such that $\gamma(u)\in B^G$. We claim that $q<\Im \eta(v_u)<q+RN_0$. Indeed, from Proposition \ref{Prop:strip}(5), one sees that for every $z\in \strip_R+a$ such that $\Im z\geq q+RN_0$ or $\Im z\leq q$, the hyperbolic distance in $\strip_R+a$ of $z$ from $B^G$ is at least $\epsilon$. For instance, if $\Im z\leq q$, then
\begin{equation*}
\begin{split}
k_{\strip_R+a}(z, B^G)&\geq k_{\strip_R+a}(a+\frac{R}{2}+i\Im z, a+\frac{R}{2}+i\Im r_1)\\&\geq k_{\strip_R+a}(a+\frac{R}{2}+i\Im q, a+\frac{R}{2}+i\Im r_1)=\epsilon,
\end{split}
\end{equation*}
where the last equality follows from a direct computation using Proposition \ref{Prop:strip}(1).

The claim we have just proved implies that $k_{\strip_R+a}(\eta(v_u), a+\frac{R}{2}+i\Im \eta(v_u))<\delta$, hence, by the triangle inequality,
\begin{equation*}
\begin{split}
 k_{\strip_R+a}(\gamma(u), a+\frac{R}{2}+i\Im \eta(v_u))&\leq k_{\strip_R+a}(\gamma(u), \eta(v_u))\\&\quad +k_{\strip_R+a}(\eta(v_u), a+\frac{R}{2}+i\Im \eta(v_u))<\epsilon+\delta.
\end{split}
\end{equation*}
Then \eqref{Eq:metr-2} follows from Proposition \ref{Prop:strip}(4).

In order to prove the last statement, suppose there exists $t\in I\setminus [u'_0,u'_1]$ such that
\[
\gamma(t)\in B_1:=\{z\in B: r_1+(\frac{N_1}{2}-N_2)R<\Im z<r_1+(\frac{N_1}{2}+N_2)R\}.
\]
 We assume that $t> u_1'$ (the  case $t<u_0'$ is similar). Note that, by definition, $\Im \gamma(u_0')=r_1$, $\Im \gamma(u_1')=r_1+RN_1$. Let $\xi:[0,1]\to \strip_R+a$ be the geodesic of $\strip_R+a$ such that $\xi(0)=\gamma(u_0')$ and $\xi(1)=\gamma(t)$. By Proposition \ref{Prop:strip}(7), for all $s\in [0,1]$,
\begin{equation}\label{Eq:xi-stima-alto}
r_1\leq \Im \xi(s)\leq \Im \gamma(t)<r_1+(\frac{N_1}{2}+N_2)R.
\end{equation}
Since $\gamma:[s_0, t]\to \Omega$ is the geodesic of $\Omega$ which joins $\xi(0)$ with $\xi(1)$, by \eqref{Eq:box-Om-strip}, there exists $s\in [0,1]$ such that $k_\Omega(\xi(s), \gamma(u'_1))<\epsilon$. Hence, by Proposition \ref{good box},
$k_{\strip_R+a}(\xi(s), \gamma(u'_1))<C\epsilon$. On the other hand, by Proposition \ref{Prop:strip}(5) and \eqref{Eq:xi-stima-alto}
\begin{equation*}
\begin{split}
k_{\strip_R+a}(\xi(s), \gamma(u'_1))&\geq k_{\strip_R+a}(a+\frac{R}{2}+i\Im\xi(s), a+\frac{R}{2}+i(r_1+RN_1))\\&=\frac{\pi}{2R}(r_1+RN_1-\Im \xi(s))>\frac{\pi}{2}(\frac{N_1}{2}-N_2)=C\epsilon,
\end{split}
\end{equation*}
a contradiction, and the proof is concluded.
\end{proof}

\section{Trajectories oscillating  to the Denjoy-Wolff point}\label{Traj}

Using the tools developed in the previous sections, we can construct examples of different slope behavior for semigroups.

\begin{proposition}\label{Prop:example-non-tg-osc}
There exists a parabolic semigroup $(\phi_t)$ in $\D$ with zero hyperbolic step such that $\phi_t(z)$ converges non-tangentially to its Denjoy-Wolff point $\tau\in\partial\D$ but $\lim_{t\to+\infty}\Arg (1-\overline{\tau}\phi_t(0))$ does not exist.
\end{proposition}
\begin{proof}
Let  $c>1$, let $D\geq D(c)$ where $D(c)$ is the constant given by Proposition \ref{Prop:local-strip-bound} and let $E\in(0, D)$.  Let $\epsilon>0$  and $C>1$ be given by Proposition \ref{good box} and let $N_0> \frac{4(C+1)\epsilon}{\pi}$.
Let $\delta>0$ and let $N>N_0>0$ be the constant given by Proposition \ref{Prop:strip}(6). Finally,  let  $N_1:=N_0-\frac{4\epsilon}{\pi}$ and $N_2:=\frac{N_0}{2}-\frac{ 2(C+1)\epsilon}{\pi}$.

 Let $\alpha_0\in (0,\pi/2)$ be such that
\begin{equation}\label{tan-alp-r}
 \tan \alpha_0\leq \min\{\frac{1}{8D}, \frac{1}{8N}\}.
\end{equation}

 Let $\chi:= 1-e^{-2(\epsilon+\delta)}<1$
and choose $\alpha_1\in (0,\alpha_0)$  such that
\begin{equation}\label{Eq:good-alpa1}
\tan \alpha_1<(1-\chi) \tan \alpha_0.
\end{equation}

In order to construct the example, we will define a domain $\Omega\subset \C$ starlike at infinity given by $\Omega=\bigcap_{j=1}^\infty \Omega_j$, where $\Omega_j:=\Omega_{-a_j, b_j,  R_j}$, where $\{b_j\}$ is an increasing  sequence of positive real numbers converging to infinity with the property that $b_1=1$,
\begin{equation}\label{Eq:choice-b}
b_{j}>b_{j-1}\max\left\{\frac{2\tan \alpha_0-\tan\alpha_1}{\tan\alpha_1}, \frac{\tan\alpha_1}{2\tan \alpha_0-\tan\alpha_1}, 4\right\}, \quad j=2,3,\ldots,
\end{equation}
and for $j=1,2,\ldots,$
\begin{equation}\label{Eq:good-choice-Ra}
\begin{split}
&R_j=2b_j\tan \alpha_0,  \\
&a_{2j} =b_{2j}\tan \alpha_1,\\
&a_{2j-1}= R_{2j-1}-b_{2j-1}\tan\alpha_1=b_{2j-1}(2\tan\alpha_0-\tan\alpha_1).
 \end{split}
 \end{equation}
Note that, by \eqref{Eq:choice-b} and \eqref{Eq:good-choice-Ra}, the sequences $\{a_{j}\}$ and $\{R_{j}-a_{j}\}$ are increasing.
In particular,
the domain $\Omega$ is starlike at infinity (see Figure \ref{fig:T}).
\begin{figure}[h]
\centering 
\begin{tikzpicture}
\draw [dotted] (4.3,-3) -- (4.3,0) node[scale=0.85]{$\bullet$};
\draw (3.5,0) node[above][scale=0.75]{$-a_{2j-1}+ib_{2j-1}$} ;
\draw [dotted] (5.2,-3) -- (5.2,0) node[scale=0.85]{$\bullet$};
\draw (6.7, 0) node[above][scale=0.75]{$-a_{2j-1}+ R_{2j-1}+ib_{2j-1}$};
\draw [dotted] (3.6,-3) -- (3.6,4) node[scale=0.85]{$\bullet$} node[above][scale=0.75]{$-a_{2j}+ib_{2j}$} ;
\draw [dotted] (9,-3) -- (9,4) node[scale=0.85]{$\bullet$} node[above][scale=0.75]{$-a_{2j}+ R_{2j}+ib_{2j}$};
\draw [dotted] (-3,-3) -- (-3,8.5) node[scale=0.85]{$\bullet$} node[above][scale=0.8]{$-a_{2j+1}+ib_{2j+1}$} ;
\draw [dotted] (10.5,-3) -- (10.5,8.5) node[scale=0.85]{$\bullet$} node[above][scale=0.75]{$-a_{2j+1}+ R_{2j+1}+ib_{2j+1}$};;
\draw [dashed] [->] (-3,-1.5) -- (11,-1.5);
\draw (5,-1.5) node[scale=0.85]{$\bullet$} ;
\draw (4.9, -1.7) node[scale=0.85][below, left]{$0$};


\draw [black,fill=gray!20] (3.6,1) -- (9,1) -- (9,3) -- (3.6,3) -- cycle;
\draw [black,fill=gray!40]  (5.5,1) -- (5.5,3) -- (7.1,3) -- (7.1,1) -- cycle;
\draw [black,fill=gray!20] (-3,5) -- (10.5,5) -- (10.5,7.5) -- (-3,7.5) -- cycle;
\draw [black,fill=gray!40]  (2.8,5) -- (2.8,7.5) -- (4.7,7.5) -- (4.7,5) -- cycle;
\draw [dashed] [->] (5,-2) -- (5,9) node[scale=0.85][left]{$it$};
\draw[->,>=latex] (10.8,4.5) to (9.8,5.5) ;
\draw (11.1,4.2) node[scale=0.85]{$B_{2j+1}$};
\draw[->,>=latex] (2.5,4.4) to (3.4,5.4) ;
\draw (2.2,4.3) node[scale=0.85]{$B_{2j+1}^+$};
\draw[->,>=latex] (7.7,0.8) to (6.9,1.4) ;
\draw (8,0.7) node[scale=0.85]{$B_{2j}^+$};
\draw (3.6, 2.8) -- (9,2.8) ;
\draw[->,>=latex](2.5, 2.2) -- (3.8,2.8);
\draw (2.2, 2.2) node[scale=0.85]{$L_{2j}^+$};
\draw (5,2.2) node{$\bullet$};
\draw (4.6,2.3) node[scale=0.85]{$iy_{2j}$} ;
\draw[->,>=latex] (2.8,1.4) to (3.9,1.4) ;
\draw (2.4,1.4) node[scale=0.85]{$B_{2j}$};
\draw (1.6,8.1) node[scale=0.85]{$\gamma([0,1))$};
\draw[->,>=latex] (2.3,8.1) to (3.4,8) ;
\draw (5,4.84) node[scale=0.85]{$\bullet$};
\draw (4.7,4.6) node[scale=0.85]{$it_{2j}$};
\draw (5,0.2) node[scale=0.85]{$\bullet$};
\draw (4.5,0.5) node[scale=0.85]{$it_{2j-1}$};

\let\pcoord\relax
\let\tcoord\relax
\foreach [count=\num] \coord in {
  (3.4,8.3),
  (3.7,5.8),
  (5.9,4.5),
  (5.9,3),
  (4.8,0)
  } {
  \ifx\pcoord\relax
   \global\let\pcoord\coord
   \path \pcoord coordinate (c1);
  \else
   \ifx\tcoord\relax
   \global\let\tcoord\coord
   \else
    \path \pcoord coordinate (p);
    \path \tcoord coordinate (t);
    \path \coord coordinate (n);
    \path ($(p)!.75!(n)$) coordinate (m);
    \path ($(t)!1cm!90:(m)$) coordinate (r);
    \path ($(t)-(p)$);
    \pgfgetlastxy{\xx}{\yy}
    \pgfmathsetmacro{\len}{.5*veclen(\xx,\yy)}
    \path ($(t)!(p)!(r)$) coordinate (rp);
    \path ($(t)!\len pt!(rp)$) coordinate (c2);
    \draw (p) .. controls (c1) and (c2) .. (t);
    \path ($(t)-(n)$);
    \pgfgetlastxy{\xx}{\yy}
    \pgfmathsetmacro{\len}{.5*veclen(\xx,\yy)}
    \path ($(t)!(n)!(r)$) coordinate (rn);
    \path ($(t)!\len pt!(rn)$) coordinate (c1);
    \global\let\pcoord\tcoord
    \global\let\tcoord\coord
   \fi
  \fi
}
\draw (t) .. controls (c1) and (n) .. (n);
\draw [dashed] (3.4,8.3) -- (3.2,9);
\draw [dashed] (4.8,0) -- (4.4,-0.5);

\end{tikzpicture}
\caption{}\label{fig:T}
\end{figure}

Moreover, the point $-a_{2j}+ib_{2j}$ belongs to $\{\zeta\in \C: \zeta=i\rho e^{\alpha_1 i}, \rho>0\}$, the boundary of $iV(\alpha_1,0)$ is contained in the left half-plane $\{z \in \C: \Re z<0\}$, while $R_{2j-1}-a_{2j-1}+ib_{2j-1}$ belongs to $\{\zeta\in \C: \zeta=i\rho e^{-\alpha_1 i}, \rho>0\}$ and the boundary of $iV(\alpha_1,0)$ is contained in the right half-plane $\{z \in \C: \Re z>0\}$. Since $(R_j-a_{j})/b_j\geq\tan\alpha_1$, this implies that  $iV(\alpha_1,+\infty)\subset \Omega$.

Hence, let  $h:\D \to \Omega$ be the Riemann map such that $h(0)=0$, $\lim_{t\to+\infty}h^{-1}(it)=1$.  The semigroup $(\phi_t)$, defined by $\phi_t(z):=h^{-1}(h(z)+it)$, has Koenigs function $h$, $\Omega=h(\D)$ and Denjoy-Wolff point $1$. Moreover, by Proposition \ref{sector-implies-convergnt}, $\phi_t(z)$ converges non-tangentially to its Denjoy-Wolff point.

Let $\tilde\gamma:[0,1)\to \D$, given by $\tilde\gamma(t)=t$, be the geodesic joining $0$ to $1$. Let $\gamma:=h\circ \tilde\gamma:[0,1)\to \Omega$. The curve $\gamma$ is a geodesic in $\Omega$, $\gamma(0)=0$, with the property that for every $M>0$ there exists $s_M\in [0,1)$ such that $\Im \gamma(s)>M$ for all $s\geq s_M$.

\smallskip

{\sl Claim A}:
\begin{enumerate}
\item there exists a sequence $\{t_m\}$ converging to $+\infty$ such that $it_m\in \gamma([0,1))$,
\item there exist $\beta>0$ and a sequence $\{t_k\}$ converging to $+\infty$ such that $it_k\not\in S_\Omega(\gamma, \beta)$.
\end{enumerate}

\smallskip

Assume that Claim A is true.  Translating in the unit disc via $h$, this implies that $\{\phi_{t_k}(0)\}$ is in the complement in $\D$ of the hyperbolic sector $S_\D(\tilde\gamma, \beta)$, thus, it is outside a fixed Stolz region of vertex $1$, while $\{\phi_{t_m}(0)\}$ converges radially to $1$. Therefore, $\lim_{t\to+\infty}\Arg(1-\phi_t(0))$ does not exist.

{\sl Proof of Claim A}.

First of all notice that, since $\tan\alpha_0\leq \frac{1}{8D}$ by \eqref{tan-alp-r}, and $b_{j}>4b_{j-1}$,
\begin{equation}\label{Eq:go-good-box}
b_j-b_{j-1}-2R_jD=b_j-b_{j-1}-4b_jD\tan\alpha_0\geq \frac{1}{2}b_j-b_{j-1}>0.
\end{equation}
Let
\[
B_j:=\{\zeta\in (\strip_{R_j}-a_j): b_{j-1}+R_jE < \Im \zeta < b_j-R_jD\}.
\]
By \eqref{Eq:go-good-box}, Proposition \ref{good box} implies that $B_j$ is a good box in $\Omega$ for the data $(c, D,E)$. Moreover, $B_j$ has width $R_j=2b_j\tan \alpha_0$ and, by \eqref{tan-alp-r}, its height is
\begin{equation*}
\begin{split}
h_j&:=b_j-b_{j-1}-R_j(D+E)>b_j-b_{j-1}-2R_j D\\&=b_j-b_{j-1}-4b_j D \tan \alpha_0\geq \frac{1}{2} b_j-b_{j-1}.
\end{split}
\end{equation*}
In particular since $b_j>4b_{j-1}$, we have by \eqref{tan-alp-r}
\begin{equation}\label{Eq:good-height-in-box}
\frac{h_j}{R_j}>\frac{1-\frac{2b_{j-1}}{b_j}}{4\tan \alpha_0}>\frac{1}{8\tan \alpha_0}\geq N.
\end{equation}

Since $\Im \gamma(s)$  converges to $+\infty$ as $s\to 1$, and $\gamma(0)=0$, it follows that there exist $0<s^0_j<t^0_j<1$ such that $\gamma(s)\in B_j$ for all $s\in (s^0_j, t^0_j)$ and $\Im \gamma(s^0_j)=b_{j-1}+R_jE$, $\Im  \gamma(t^0_j)=b_j-R_jD$. Therefore, by \eqref{Eq:good-height-in-box}, we can find  $s_j^0<s_j<t_j<t_j^0$ such that $\Im \gamma(t_j)-\Im \gamma(s_j)>NR_j$.

Hence, by Corollary \ref{cor:good-box-estim} (with $B=B_j, a=-a_j, R=R_j$) there exists  $r_j\in (b_{j-1}+R_{j-1}E, b_j-R_j(D+N_1))$ such that  for all $u\in (s_j,t_j)$ such that $r_j<\Im \gamma(u)<r_j+N_1R_j$ (recalling that $\chi= 1-e^{-2(\epsilon+\delta)})$,
\begin{equation}\label{Eq:gamma-in_Bntosc}
\gamma(u)\in B_j^+:=\{z\in B_j: |\Re z+a_j-\frac{R_j}{2}|\leq  \frac{\chi}{2}R_j\}.
\end{equation}
Using \eqref{Eq:good-choice-Ra}, it is easy to see that  $|a_j-\frac{R_j}{2}|=b_j(\tan \alpha_0-\tan \alpha_1)$, while  $R_j/2=b_j \tan \alpha_0$. Hence, by \eqref{Eq:good-alpa1},
\begin{equation*}
|a_j-\frac{R_j}{2}|-\frac{\chi}{2}R_j=b_j((1-\chi)\tan \alpha_0-\tan \alpha_1)>0.
\end{equation*}
Therefore,  $it\not\in B^+_j$ for all $t>0$ such that $it\in B_j$.

Moreover, since $-a_{j}+\frac{R_j}{2}=(-1)^jb_{j}(\tan \alpha_0-\tan \alpha_1)$, it follows that $B_j^+\subset \{z\in \C: (-1)^j\Re z>0\}$. Hence, if $u_j\in (s_j,t_j)$ is such that $r_j<\Im \gamma(u_j)<r_j+N_1R_j$, we have  $\Re \gamma(u_{2j})>0$ and $\Re \gamma(u_{2j-1})<0$, $j=1,2,\ldots$.

Since $\gamma$ is continuous, it follows that there exists a sequence $\{t_m\}$ converging to $+\infty$ such that $it_m\in \gamma([0,1))$ for all $m\in \N$: Part (1) of Claim A is proved.

As for Part (2) of Claim A,  let $u_{2j}\in (s_{2j},t_{2j})$ be such that $\Im \gamma(u_{2j})=r_{2j}+\frac{N_1}{2}R_{2j}$.  Note that $r_{2j}<\Im \gamma(u_{2j})<r_{2j}+N_1R_{2j}$. Let
\[
y_{2j}:=\Im \gamma(u_{2j})=r_{2j}+\frac{N_1}{2}R_{2j}.
\]
  Let $x_{2j}\in (0,R_{2j}/2)$ be such that  $-x_{2j}-a_{2j}+\frac{R_{2j}}{2}=\frac{\chi}{2}R_{2j}$. Note that by \eqref{Eq:good-choice-Ra} and \eqref{Eq:good-alpa1},
\[
x_{2j}=b_{2j} [(1-\chi)\tan\alpha_0-\tan \alpha_1]>0.
\]
By Proposition \ref{Prop:strip}(3) (with $R=R_{2j}, a=-a_{2j}$), the curve $\eta:(-\frac{R_{2j}}{2},\frac{R_{2j}}{2})\ni s\mapsto s-a_{2j}+\frac{R_{2j}}{2}+iy_{2j}$ is a geodesic of $\strip_{R_{2j}}-a_{2j}$, and by Proposition \ref{Prop:strip}(4),
\begin{equation*}
\begin{split}
k_{\strip_{R_{2j}}-a_{2j}}(B_{2j}^+, iy_{2j})&= k_{\strip_{R_{2j}}-a_{2j}}(\eta(-x_{2j}), iy_{2j})=k_{\strip_{R_{2j}}-a_{2j}}(\eta(-x_{2j}), \eta(a_{2j}-\frac{R_{2j}}{2}))\\
&\geq \frac{1}{2}\log\frac{R_{2_j}-2x_{2j}}{R_{2j}-2 |a_{2j}-\frac{R_{2j}}{2}|}=\log\frac{\chi \tan\alpha_0+\tan\alpha_1}{\tan\alpha_1}=:\tilde\beta_1>0,
\end{split}
\end{equation*}
where the inequality follows  from the left inequality in Proposition \ref{Prop:strip}(3) and the last equality follows from \eqref{Eq:good-choice-Ra}.

By Proposition \ref{good box},  $k_{\Omega}(B_{2j}^+, iy_{2j})\geq \frac{1}{C}k_{\strip_{R_{2j}}-a_{2j}}(B_{2j}^+, iy_{2j})$. Let $\beta_1:=\frac{\tilde\beta_1}{C}$. The previous estimate and \eqref{Eq:gamma-in_Bntosc} imply then that for all $t\in (s_{2j}, t_{2j})$  such that $r_{2j}<\Im \gamma(t)<r_{2j}+N_1R_{2j}$
\begin{equation}\label{Eq:good-part-estimabeta}
k_{\Omega}(\gamma(t), iy_{2j})\geq \beta_1.
\end{equation}
Now, let $W_{2j}^+:=\{z\in \Omega: \Im z\geq r_{2j}+(\frac{N_1}{2}+N_2)R_{2j}\}$ and $W_{2j}^-:=\{z\in \Omega: \Im z\leq r_{2j}+(\frac{N_1}{2}-N_2)R_{2j}\}$. Assume $t\in [0,1)$ and $\gamma(t)\in W^+_{2j}$. Let $L_{2j}^+:=\{z\in B_{2j}: \Im z=r_{2j}+(\frac{N_1}{2}+N_2)R_{2j}\}$. Hence, by Proposition \ref{good box} and Proposition \ref{Prop:strip}(5)
\begin{equation*}
\begin{split}
k_\Omega(\gamma(t), iy_{2j})&\geq k_\Omega(W^+_{2j}, iy_{2j})=
  k_{\Omega}(L^+_{2j}, iy_{2j})\geq \frac{1}{C}k_{\strip_{R_2j}-a_{2j}}(L^+_{2j}, iy_{2j})\\&\geq \frac{1}{C}k_{\strip_{R_2j}-a_{2j}}(L^+_{2j}, R_{2j}-a_{2j}+iy_{2j})\\&=\frac{1}{C}k_{\strip_{R_2j}-a_{2j}}(L^+_{2j}, R_{2j}-a_{2j}+i(r_{2j}+\frac{N_1}{2}R_{2j}))=\frac{\pi N_2}{2C}.
\end{split}
\end{equation*}
A similar computation shows that $k_\Omega(\gamma(t), iy_{2j})\geq \frac{\pi N_2}{2C}$ for all $t\in [0,1)$ and $\gamma(t)\in W^-_{2j}$.

Let $\beta:=\min\{\beta_1, \frac{\pi N_2}{2C}\}$ and let $u_1^{2j}:=\min \{t\in [s_{2j},t_{2j}]:  \Im \gamma(t)\geq r_{2j}\}$ and $u_2^{2j}:=\max \{t\in [s_{2j},t_{2j}]:  \Im \gamma(t)\leq r_{2j}+N_1R_{2j}\}$.

By Corollary \ref{cor:good-box-estim}, for every $t\in I\setminus [u_1^{2j},u_2^{2j}]$ it follows that $\gamma(t)\in W_{2j}^+\cup W_{2j}^-$, hence $k_\Omega(\gamma(t), iy_{2j})\geq \beta$ for what we already proved. On the other hand, if $t\in [s_{2j},t_{2j}]$ then either $\gamma(t)\in W_{2j}^+\cup W_{2j}^-$ and hence $k_\Omega(\gamma(t), iy_{2j})\geq \beta$, or $r_{2j}<\Im \gamma(t)<r_{2j}+N_1R_{2j}$ and hence $k_\Omega(\gamma(t), iy_{2j})\geq \beta$ by \eqref{Eq:good-part-estimabeta}.

Therefore the sequence $\{iy_{2j}\}$ satisfies the requirement of Claim A Part (2).
\end{proof}

For $\alpha\geq 1$, let
\[
Z_\alpha:=\{z\in \C: |\Re z|^\alpha<\Im z\}.
\]
As it is shown in Proposition \ref{sector-implies-convergnt}, if $(\phi_t)$ is a non-elliptic semigroup in $\D$ with universal model $(\C, h, z+it)$ and $Z_1+p\subset h(\D)$ for some $p \in h(\D)$, then $\phi_t(z)$ converges non-tangentially to its Denjoy-Wolff point. In the following proposition, we show that $\alpha=1$ is the best we can have:

\begin{proposition}\label{Prop:tang-conve-with-para}
Let $\alpha>1$. Then there exists a parabolic  semigroup $(\phi_t)$ of $\D$ such that $\phi_t$ does not converge non-tangentially to its Denjoy-Wolff point, and such that if $h$ is its Koenigs function, then $Z_\alpha\subset h(\D)$.
\end{proposition}

\begin{proof} Let  $c>1$, let $D\geq D(c)$ be the constant given by Proposition \ref{Prop:local-strip-bound} and let $E\in(0, D)$. Let $C>1$ be given by Proposition \ref{good box}. Let $\{N_j^0\}$ be an increasing sequence of positive numbers, converging to $+\infty$, such that $N_j^0>\max\{\frac{4(C+1)\epsilon}{\pi}, 2D\}$ for all $j$. Let $\delta>0$. For every $j$, let $N_j>N_j^0$ be the constant given in Proposition \ref{Prop:strip}(6), relative to the pair $(N^0=N_j^0, \delta)$.

Let $\beta\in (\frac{1}{\alpha},1)$. Let  $\{R_j\}$ be a sequence of positive real numbers, converging to $\infty$, such that for all $j=1,2,\ldots$,
\begin{equation*}
\begin{split}
& R_0^{\beta-1}<\frac{1}{2},\\
& R_j>R_{j-1}^{\beta\alpha},\\
& R_j^{\alpha\beta-1}>N_j+1.
\end{split}
\end{equation*}
For $j=0,1,2,\ldots$ we set
\begin{equation*}
\begin{split}
& a_j:=R_j^\beta,\\
& b_j:=a_j^\alpha=R_j^{\alpha\beta}.
\end{split}
\end{equation*}
Note that
\begin{equation}\label{Eq:behavior-constants}
\begin{split}
& \lim_{j\to +\infty} \frac{R_j}{a_j}=+\infty,\\
& \frac{b_j-b_{j-1}}{R_j}> R_j^{\alpha\beta-1}-1>N_j>N_j^0>2D.
\end{split}
\end{equation}
Let $\Omega_j:=\Omega_{-a_j, b_j,  R_j}$ and $\Omega:=\cap_{j=0}^\infty \Omega_j$.

Note that $\Omega$ is starlike at infinity and $a_j=R_j^\beta>R_{j-1}^{\beta}=a_{j-1}$. Moreover,
\begin{equation}\label{Eq:R_j-tiene-testa-aj}
R_j-2a_j=R_j-2R_j^\beta=R_j(1-2R_j^{\beta-1})>R_j(1-2R_0^{\beta-1})>0.
\end{equation}
Therefore, $Z_\alpha\subset \Omega$.

We let $h:\D \to \Omega$ be the Riemann map such that $h(0)=i$. We define the semigroup $\phi_t:=h^{-1}(h(z)+it)$. We can assume that $1$ is its Denjoy-Woff point. We prove that there exists a sequence $\{t_k\}$, converging to $+\infty$, such that $k_\D([0,1), \phi_{t_k}(0))\to +\infty$ when $k$ tends to $\infty$, which means that the sequence $\{\phi_{t_k}(0)\}$ converges tangentially to $1$.

Let $\gamma_0:=h([0,1))$. The previous condition is equivalent to find a sequence $\{t_k\}$ converging to $+\infty$ such that $k_{\Omega}(\gamma_0, it_k)\to +\infty$ when $k$ tends to $\infty$.

In order to prove this, we note that by \eqref{Eq:behavior-constants}, $b_j-b_{j-1}>R_j N_j>2R_jD$. Hence,
\[
B_j:=\{\zeta\in (\strip_{R_j}-a_j): b_{j-1}+R_j E<\Im \zeta < b_j-R_j D\},
\]
is a good box for $\Omega$ for the data $(c,D,E)$, with width $R_j$ and height $b_j-b_{j-1}$ (see Proposition~\ref{good box}).

Since $\gamma_0:[0,1)\to \Omega$ has the property that $\gamma_0(0)=i$ and $\Im \gamma_0(t)\to +\infty$ as $t\to+\infty$, we can find $0<s_j^0<t_j^0<1$ such that $\gamma_0(s)\in B_j$ for all $s\in (s_j^0, t_j^0)$ and $\Im \gamma_0(s_j^0)=b_{j-1}+R_jE$, $\Im \gamma_0(t_j^0)=b_{j}-R_jD$.

Let
\[
N^1_j:=N^0_j-\frac{4\epsilon}{\pi}, \quad N^2_j:=\frac{N^0_j}{2}-\frac{2(C+1)\epsilon}{\pi}, \quad \chi:= 1-e^{-2(\epsilon+\delta)}.
\]

By Corollary  \ref{cor:good-box-estim} (with $B=B_j, a=-a_j, R=R_j$) there exists  $r_j\in (b_{j-1}+R_{j}E, b_j-R_j(D+N_j^1))$ such that  for all $u\in (s_j,t_j)$ such that $r_j<\Im \gamma(u)<r_j+N^1_jR_j$,
\begin{equation}\label{Eq:gamma-in_Bntosc}
\gamma(u)\in B_j^+:=\{z\in B_j: |\Re z+a_j-\frac{R_j}{2}|\leq  \frac{\chi}{2}R_j)\}.
\end{equation}

Moreover, let $u_0^{j}:=\min \{t\in [s_{j},t_{j}]:  \Im \gamma(t)\geq r_{j}\}$ and $u_1^{j}:=\max \{t\in [s_{j},t_{j}]:  \Im \gamma(t)\leq r_{j}+N_j^1R_{j}\}$. It follows from Corollary  \ref{cor:good-box-estim} that for every $t\in I\setminus [u^j_0,u^j_1]$
\begin{equation}\label{Eq:fuera-buena-caja-j}
\gamma(t)\not\in \{z\in B_j: r_j+(\frac{N_j^1}{2}-N_j^2)R_j<\Im z<r_j+(\frac{N_j^1}{2}+N_j^2)R_j\}.
\end{equation}

Note that, if $z\in B_j^+$ then
\[
\Re z\geq \frac{R_j}{2}(1-\chi)-a_j=\frac{1-\chi}{2}R_j-R^\beta_j,
\]
and, since $\frac{1-\chi}{2}>0$ and $\beta<1$, for every $M>0$ there exists $j_M$ such that $B_j^+\subset \{\zeta\in \C: \Re \zeta>M\}$ for all $j\geq j_M$.

Let $u_{j}\in (u_1^{j},u_2^{j})$ be such that $\Im \gamma(u_{j})=r_{j}+\frac{N^1_j}{2}R_{j}$.  Note that $r_{j}<\Im \gamma(u_{j})<r_{j}+N^1_jR_{j}$. Let
\[
y_{j}:=\Im \gamma(u_{j})=r_{j}+\frac{N^1_j}{2}R_{j}.
\]
  Let $x_{j}\in (0,R_{j}/2)$ be such that  $-x_{j}-a_{j}+\frac{R_{j}}{2}=\frac{\chi}{2}R_{j}$. Note that, for $j$ sufficiently large, $x_j>0$.

  By Proposition \ref{Prop:strip}(2) (with $R=R_{j}, a=-a_{j}$), the curve $\eta:(-\frac{R_{j}}{2},\frac{R_{j}}{2})\ni s\mapsto s-a_{j}+\frac{R_{j}}{2}+iy_{j}$ is a geodesic of $\strip_{R_{j}}-a_{j}$, and by Proposition \ref{Prop:strip}(4),
\begin{equation*}
\begin{split}
k_{\strip_{R_{j}}-a_{j}}(B_j^+, iy_{j})&= k_{\strip_{R_{j}}-a_{j}}(\eta(-x_{j}), iy_{j})=k_{\strip_{R_{j}}-a_{j}}(\eta(-x_{j}), \eta(a_{j}-\frac{R_{j}}{2}))\\
&\geq \frac{1}{2}\log\frac{R_{j}-2x_{j}}{R_{j}-2 (\frac{R_{j}}{2}-a_{j})}=\frac{1}{2}\log \frac{\chi R_j+2a_j}{2a_j}\simeq \frac{1}{2}\log \frac{R_j}{a_j},
\end{split}
\end{equation*}
where the last inequality follows  from \eqref{Eq:R_j-tiene-testa-aj} and the inequalities in Proposition \ref{Prop:strip}(2).

By Proposition  \ref{good box},  $k_{\Omega}(B_j^+, iy_{j})\geq \frac{1}{C}k_{\strip_{R_{2j}}-a_{j}}(B_j^+, iy_{j})$.

The previous estimate and \eqref{Eq:gamma-in_Bntosc} imply then that for all $M>0$ there exists $j_0^M$ such that for all $j\geq j_0^M$ and $t\in (s_{j}, t_{j})$  such that $r_{j}<\Im \gamma(t)<r_{j}+N_j^1R_{j}$
\begin{equation}\label{Eq:good-part-estimainf}
k_{\Omega}(\gamma(t), iy_{j})>M.
\end{equation}
Now, let $W_{j}^+:=\{z\in \Omega: \Im z\geq r_{j}+(\frac{N_j^1}{2}+N_j^2)R_{j}\}$ and $W_{j}^-:=\{z\in \Omega: \Im z\leq r_{j}+(\frac{N_j^1}{2}-N_j^2)R_{j}\}$. Assume $t\in [0,1)$ and $\gamma(t)\in W^+_{j}$. Let $L_{j}^+:=\{z\in B_{j}: \Im z=r_{j}+(\frac{N_j^1}{2}+N_j^2)R_{j}\}$. Hence, by Proposition  \ref{good box} and Proposition \ref{Prop:strip}(5),
\begin{equation*}
\begin{split}
k_\Omega(\gamma(t), iy_{j})&\geq k_\Omega(W^+_{j}, iy_{j})=
  k_{\Omega}(L^+_{j}, iy_{j})\geq \frac{1}{C}k_{\strip_{R_j}-a_{j}}(L^+_{j}, iy_{j})\\&\geq \frac{1}{C}k_{\strip_{R_j}-a_{j}}(L^+_{j}, R_{j}-a_{j}+iy_{j})\\&=\frac{1}{C}k_{\strip_{R_j}-a_{j}}(L^+_{j}, R_{j}-a_{j}+i(r_{j}+\frac{N_j^1}{2}R_{j}))=\frac{\pi N_j^2}{2C}.
\end{split}
\end{equation*}
A similar computation shows that $k_\Omega(\gamma(t), iy_{j})\geq \frac{\pi N_j^2}{2C}$ for all $t\in [0,1)$ and $\gamma(t)\in W^-_{j}$.

Note that $\lim_{j\to +\infty}\frac{\pi N_j^2}{2C}=+\infty$.  Hence, for every $M>0$ there exists $j_1^M$ such that for all $j\geq j_1^M$ and all $t\in [0,1)$ such that $\gamma(t)\in W_j^+\cup W_j^-$, we have $k_\Omega(iy_j, \gamma(t))>M$.

By \eqref{Eq:fuera-buena-caja-j}, and \eqref{Eq:good-part-estimainf} this means that for every $M>0$, and for every $j\geq \max\{j_0^M, j_1^M\}$ we have $k_{\Omega}(iy_j, \gamma(t))>M$ for all $t\in [0,1)$. The proof is completed.
\end{proof}


\begin{thebibliography}{99}
\bibitem{Abate} M. Abate, {\sl Iteration theory of holomorphic maps on taut manifolds}, Mediterranean Press, Rende, 1989.
\bibitem{AhElReSh99} D. Aharonov, M. Elin, S. Reich and D.  Shoikhet, {\sl Parametric representation of semi-complete vector fields on the unit balls in $\C^n$ and in Hilbert space}, Rend. Mat. Acc. Lincei, \textbf{10} (1999), 229--253.
\bibitem{BrAr} L. Arosio, F. Bracci, {\sl Canonical models for holomorphic iteration}. Trans. Amer. Math. Soc., \textbf{368}, (2016), 3305--3339.
\bibitem{Berkson-Porta} E. Berkson, H. Porta, {\sl Semigroups of
holomorphic functions and composition operators}. Michigan
Math. J., \textbf{25} (1978), 101--115.
\bibitem{Bet}  D. Betsakos, {\sl On the asymptotic behavior of the trajectories of semigroups of holomorphic functions}. J. Geom. Anal., \textbf{26} (2016),  557--569.
\bibitem{BCD}  F. Bracci, M. D. Contreras, S. D\'iaz-Madrigal, {\sl Topological invariants for semigroups of holomorphic self-maps of the unit disc},  J. Math. Pures Appl., \textbf{107},  (2017), 78--99.
\bibitem{BCDG} F. Bracci, M. D. Contreras, S. D\'iaz-Madrigal, H. Gaussier, {\sl  A characterization of orthogonal convergence in simply connected domains}, J. Geom. Anal. (2018). https://doi.org/10.1007/s12220-018-00109-8
\bibitem{CD} M. D. Contreras and S. D\'iaz-Madrigal, {\sl Analytic flows in the unit disk: angular derivatives and boundary
fixed points}. Pacific J. Math., \textbf{222} (2005), 253--286.
\bibitem{CDP} M. D. Contreras, S. D\'iaz-Madrigal,  Ch. Pommerenke, {\sl Some remarks on the Abel equation in the
unit disk}, J. London Math. Soc. (2), \textbf{75}, (2007), 623--634.
\bibitem{CDG} M. D. Contreras, S. D\'iaz-Madrigal, P. Gumenyuk, {\sl Slope problem for trajectories of holomorphic semigroups in the unit disk}. Comput. Methods Funct. Theory, \textbf{15} (2015), 117--124.
\bibitem{CL} E.F. Collingwood, A.J. Lohwater, {\sl The theory of cluster sets}, Cambridge Tracts in Mathematics and Mathematical Physics, No. 56 Cambridge Univ. Press, Cambridge, 1966.
\bibitem{Cowen} C. C. Cowen, {\sl Iteration and the solution of functional equations for functions analytic in the unit disk}, Trans. Amer. Math. Soc., \textbf{265} (1981), 69--95.
\bibitem{EKRS} M. Elin, D. Khavinson, S. Reich, D. Shoikhet, {\sl Linearization models for parabolic dynamical systems via Abel's functional equation}. Ann. Acad. Sci. Fenn., 35, (2010), 439-472.
\bibitem{ES} M. Elin, D. Shoikhet, {\sl Linearization models for complex dynamical systems}, topics in Univalent Functions, Functional Analysis and Semigroup Theory, Birkh\"auser, Basel, 2010, 265 pp.
\bibitem{EYRS} M. Elin, F. Yacobzon, S. Reich, D. Shoikhet, {\sl Asymptotic behavior of one-parameter semigroups and rigidity of holomorphic generators}, Complex Anal. Oper. Theory, 2, (2008), 55-86.
\bibitem{Ghys} E. Ghys, P. de La Harpe, {\sl Sur les groupes hyperboliques d'apr\`es Mikhael Gromov}, Progress in Mathematics, 83, Birkh\"auser.
\bibitem{K} G. Kelgiannis, {\sl Trajectories of semigroups of holomorphic functions and harmonic measure}, J. Math. Anal. Appl., 474, (2019), 1364-1374.
\bibitem{Pommerenke}Ch. Pommerenke, {\sl Univalent functions}, Vandenhoeck \&
Ruprecht, G\"{o}ttingen, 1975.
\bibitem{Pommerenke2}Ch. Pommerenke, {\sl Boundary behaviour of conformal mappings}, Springer-Verlag, 1992.
\bibitem{Shb} D. Shoikhet, {\sl Semigroups in geometrical function theory}. Kluwer Academic Publishers, Dordrecht, 2001.
\bibitem{Siskakis-tesis} A. G. Siskakis, {\sl Semigroups of composition
operators on the space of analytic functions}. Contemp. Math. 213, (1998), 229-252.
\end{thebibliography}
\end{document}